\documentclass[12pt, reqno]{amsart}
\usepackage{amsmath, amsthm, amscd, amsfonts, amssymb, graphicx, color}
\usepackage[bookmarksnumbered, colorlinks, plainpages]{hyperref}

\usepackage{stmaryrd,latexsym,bm}
\usepackage[latin1]{inputenc}
\usepackage{t1enc}
\usepackage{cite}

\newtheorem{theorem}{Theorem}[section]
\newtheorem{lemma}[theorem]{Lemma}
\newtheorem{proposition}[theorem]{Proposition}
\newtheorem{corollary}[theorem]{Corollary}
\theoremstyle{definition}
\newtheorem{definition}[theorem]{Definition}
\theoremstyle{remark}
\newtheorem{remark}[theorem]{Remark}
\newtheorem{example}[theorem]{Example}
\numberwithin{equation}{section}

\newcommand{\supp}{\mathop{\mathrm{supp}\,}\nolimits}
\newcommand{\id}{\mathop{\mathrm{id}}\nolimits}
\newcommand{\dint}{\mathrm{d}}
\newcommand{\Dd}{\mathrm{D}}
\newcommand{\tr}[1]{\mathrm{tr}\rule[-0.5ex]{0ex}{1.2ex}_{#1}}
\newcommand{\ve}{\ensuremath{\varepsilon}}
\newcommand{\sul}{\underline{\mathfrak{s}}}
\newcommand{\sol}{\overline{\mathfrak{s}}}
\newcommand{\upind}[1]{\sol\left(\bm{#1}\right)}
\newcommand{\lowind}[1]{\sul\left(\bm{#1}\right)}
\newcommand{\beq}{\begin{equation}}
\newcommand{\eeq}{\end{equation}}

\newcommand{\bli}{\begin{list}{}{\labelwidth6mm\leftmargin8mm}}
\newcommand{\eli}{\end{list}}
\newcommand{\paren}[1]{\bm{(}#1\bm{)}}
\newcommand{\whole}[1]{\ensuremath\left\lfloor #1 \right\rfloor}

\newcommand{\gseq}[3]{\ensuremath {(#1_#2)_{#2 \in #3}}}
\newcommand{\ds}{\displaystyle}

\newcommand{\Sisih}{\ensuremath{\Sigma}_{\bm{\sigma},\bm{h}}}

\newcommand{\real}{\ensuremath{\mathbb{R}}}
\newcommand{\rn}{\ensuremath{\mathbb{R}^n}}
\newcommand{\nat}{\ensuremath{\mathbb{N}}}
\newcommand{\no}{\ensuremath{\mathbb{N}_0}}
\newcommand{\non}{\ensuremath{\mathbb{N}^n_0}}

\newcommand{\comp}{\ensuremath{\mathbb{C}}}

\newcommand{\zn}{\ensuremath{\mathbb{Z}^n}}
\newcommand{\HH}{\ensuremath{\mathbb H}}

\newcommand{\SpRn}{\mathcal{S}'(\rn)}
\newcommand{\SRn}{\mathcal{S}(\rn)}
\newcommand{\B}{B^s_{p,q}}
\newcommand{\Bsi}{B^{\bm{\sigma}}_{p,q}}

\newcommand{\Bsg}{{\mathbb{B}}^{\bm{\sigma}}_{p,q}(\Gamma)}
\newcommand{\Bsge}{{\mathbb{B}}^{\bm{\sigma}}_{p_1,q_1}(\Gamma)}
\newcommand{\Btau}{B^{\bm{\tau}}_{p,q}}
\newcommand{\Btgz}{{\mathbb{B}}^{\bm{\tau}}_{p_2,q_2}(\Gamma)}

\newcommand{\Bsih}{B^{\bm{\sigma} \bm{h}^{1/p}\paren{n}^{1/p}}_{p,q}\!(\rn)}
\newcommand{\Bsihe}{B^{\bm{\sigma} \bm{h}^{1/p_1}\paren{n}^{1/p_1}}_{p_1,q_1}\!(\rn)}
\newcommand{\Btauhz}{B^{\bm{\tau} \bm{h}^{1/p_2}\paren{n}^{1/p_2}}_{p_2,q_2}\!(\rn)}

\newcommand{\egX}{\mbox{${\mathcal E}_{\!\mathsf G}^X$}}
\newcommand{\egv}[1]{\mbox{${\mathcal E}_{\!\mathsf G}^{#1}$}}
\newcommand{\envg}{\mbox{${\mathfrak E}\rule[-0.5ex]{0em}{2.3ex}_{\!\mathsf G}$}}
\newcommand{\uGindex}[1]{u\rule[-0.5ex]{0ex}{1.1ex}_{\mathsf G}^{#1}}

\begin{document}
\setcounter{page}{1}

\title{Embeddings of Besov spaces on fractal $h$-sets}

\author[A.M. Caetano, D.D. Haroske]{Ant{\'o}nio M. Caetano$^1$ and Dorothee D. Haroske$^2$$^{*}$}

\address{$^{1}$ Center for R\&D in Mathematics and Applications, Department of Mathematics, University of Aveiro, 3810-193 Aveiro, Portugal}
\email{\textcolor[rgb]{0.00,0.00,0.84}{acaetano@ua.pt}}

\address{$^{2}$ Institute of Mathematics, Friedrich-Schiller-University Jena, 07737 Jena, Germany}
\email{\textcolor[rgb]{0.00,0.00,0.84}{dorothee.haroske@uni-jena.de}}



\date{}

\begin{abstract}
Let $\Gamma$ be a fractal $h$-set and $\Bsg$ be a trace space of Besov
type defined on $\Gamma$. While we dealt in \cite{Cae-env-h} with
growth envelopes of such spaces mainly and investigated the existence
of traces in detail in \cite{CaH-3}, we now study continuous embeddings between
different spaces of that type on $\Gamma$. We obtain necessary and sufficient
conditions for such an embedding to hold, and can prove in some cases
complete characterisations. It also includes the
situation when the target space is of type $L_r(\Gamma)$ and, as a by-product, under mild assumptions on the $h$-set $\Gamma$ we obtain the exact conditions on $\sigma$, $p$ and $q$ for which the trace space $\Bsg$ exists. We can also
refine some embedding results for spaces of generalised smoothness on $\rn$.
\end{abstract} \maketitle

\section*{Introduction}
We study in this paper continuous embeddings of trace spaces on a
fractal $h$-set $\Gamma$. It is the natural continuation and so far
final step of our project to characterise traces taken on some $h$-set
$\Gamma$, which led to the papers \cite{Cae-env-h,CaH-3}
already. Questions of that type are of particular interest in view of boundary value problems
of elliptic operators, where the solutions belong to some appropriate Besov (or Sobolev) space. One standard method is to start with assertions about traces on hyperplanes and then to transfer these
results to bounded domains with sufficiently smooth boundary
afterwards. Further studies may concern compactness or regularity results,
leading to the investigation of spectral properties. However, when it comes to irregular (or fractal) boundaries, one has to circumvent a lot of difficulties following that way, such that another method turned out to be more appropriate. This approach was proposed by Edmunds and Triebel in connection with smooth boundaries first in \cite{ET} and then extended to fractal $d$-sets in \cite{ET7,ET6,T-Frac}. Later the setting of $d$-sets was extended to $(d,\Psi)$-sets by Moura \cite{moura-diss} and finally to the more general $h$-sets by Bricchi \cite{bricchi-diss}. \\
The idea is rather simple to describe, but the details are much more complicated: at first one determines the trace spaces of certain Besov (or Sobolev) spaces as precisely as possible, studies (compact) embeddings of such spaces into appropriate target spaces together with their entropy and approximation numbers afterwards, and finally applies Carl's or Weyl's inequalities to link eigenvalues and entropy or approximation numbers. If one is in the lucky situation that, on one hand, one has atomic or wavelet decomposition results for the corresponding spaces, and, on the other hand, the irregularity of the fractal can be characterised by its local behaviour (within `small' cubes or balls), then there is some chance to shift all the arguments to appropriate sequence spaces which are usually easier to handle. This explains our preference for fractal $h$-sets and Besov spaces. But still the problem is not so simple and little is known so far. \\
Dealing with spaces on $h$-sets we refer to \cite{CL-2,CL-3,K-Z,lopes-diss}, and, probably closest to our approach here, 
 \cite[Ch.~7]{T-F3}. There it turns out that one first needs a sound knowledge about the existence and quality of the corresponding trace spaces. Returning to the first results in that respect in \cite{bricchi-diss}, see also \cite{bricchi-fsdona,bricchi-3,bricchi-4}, we found that the approach can (and should) be extended for later applications. More precisely, for a positive continuous and non-decreasing
function $h : (0,1]\to\real$ (a gauge function) with $\lim_{r\to 0} h(r)=0$, a non-empty compact set $\Gamma\subset\rn$ is called {\em $h$-set} if there exists a finite Radon measure $\mu$ in $\rn$ with $\supp\mu=\Gamma$ and 
\[
\mu(B(\gamma,r)) \sim h(r),\qquad r\in(0,1],\ \gamma\in\Gamma,
\]
see  also \cite[Ch.~2]{Rogers} and \cite[p.~60]{mattila}. In the special case $h(r)=r^d$, $0< d< n$, $\Gamma$ is called $d$-set (in the sense of \cite[Def.~3.1]{T-Frac}, see also \cite{JW-1,mattila} -- be aware that this is different from the sense used in \cite{falconer}). Recall that some self-similar fractals are outstanding examples of $d$-sets; for instance,
the usual (middle-third) Cantor set in $\real^1$ is a $d$-set for $d = \ln 
2/\ln 3$, and the Koch curve in $\real^2$ is a $d$-set for $d=\ln 4/\ln
3$. \\
The trace is defined by completion of pointwise traces of $\varphi\in\mathcal{S}(\rn)$, assuming that 
for $0<p<\infty$ we have in addition $\ds \| \varphi\raisebox{-0.5ex}[1ex][1ex]{$\big|_{\Gamma}$}~ | L_p(\Gamma) \| \lesssim \  \| \varphi | B^t_{p,q}(\rn)\|$ for suitable parameters $t\in\real$ and $0<q<\infty$. 
In case of a compact $d$-set $\Gamma$, $0<d<n$, this results in
\beq
\tr{\Gamma} B^{\frac{n-d}{p}}_{p,q}(\rn) = L_p(\Gamma)\quad \text{if} \quad 0<q\leq\min (p,1)
\label{Lp-trace-d}
\eeq
and, for $s>\frac{n-d}{p}$,
\[
\tr{\Gamma}\B(\rn)  = {\mathbb B}^{s-\frac{n-d}{p}}_{p,q}(\Gamma),
\]
see \cite{T-Frac} with some later additions in \cite{T-func,T-F3}. Here $\B(\rn)$ are the usual Besov spaces defined on $\rn$. In the classical case $d=n-1$, $0<p<\infty$, $0<q\leq\min (p,1)$, \eqref{Lp-trace-d} reproduces -- up to our compactness  requirement -- the well-known trace result $\tr{\real^{n-1}} B^{\frac1p}_{p,q}(\rn) = L_p(\real^{n-1})$. \\
In case of $h$-sets $\Gamma$ one needs to consider Besov spaces of generalised smoothness $\Bsi(\rn)$ which naturally extend $\B(\rn)$: instead of the smoothness parameter $s\in\real$ one now admits sequences $\bm{\sigma} = (\sigma_j)_{j\in\no}$ of positive numbers which satisfy $ \sigma_j  \sim  \sigma_{j+1}$, $j\in\no$. Such spaces are special cases of the scale $B^{\bm{\sigma},N}_{p,q}(\rn)$ which is well studied by different approaches, namely the interpolation one \cite{merucci,C-F} and the theory developed independently by Gol'dman and Kalyabin in the late 70's and early 80's of the last century, see the
survey \cite{Kal-Liz} and the appendix \cite{lizorkin} which cover the
extensive (Russian) literature at that time. We shall rely on the Fourier-analytical
approach as presented in \cite{F-L}. It turns out that the classical smoothness $s\in\real$ has to be replaced by certain regularity indices $\upind{\sigma}$, $\lowind{\sigma}$ of $\bm{\sigma}$. In case of $\bm{\sigma}= (2^{js})_j$ the spaces $\Bsi(\rn)$ and $\B(\rn)$ coincide and $\upind{\sigma}=\lowind{\sigma}=s$. \\
Dealing with traces on $h$-sets $\Gamma$ in a similar way as for $d$-sets, one obtains
\[
\tr{\Gamma} \Btau(\rn)= \Bsg,
\]
where the sequence $\bm{\tau}$ (representing smoothness) depends on $\bm{\sigma}$, $h$ (representing the geometry of $\Gamma$) and the underlying $\rn$; in particular, with $\bm{h} :=
\left(h(2^{-j})\right)_j$, $ \bm{h_{p}} = \left(h(2^{-j})^\frac1p\
2^{j\frac{n}{p}}\right)_j$, the counterpart of \eqref{Lp-trace-d} reads as
\[
\tr{\Gamma} B^{\bm{h_{p}}}_{p,q}(\rn) = L_p(\Gamma), \qquad 0<q\leq\min (p,1). 
\]

These results were already obtained in \cite{bricchi-3} under some
additional restrictions. In \cite{Cae-env-h} we studied sufficient
conditions for the existence of such traces again (in the course of
dealing with growth envelopes, characterising some singularity
behaviour) and returned in \cite{CaH-3} to the subject to obtain
`necessary' conditions, or, more precisely, conditions for the
non-existence of traces. This problem is closely connected with the
so-called {\em dichotomy}; we refer to \cite{CaH-3} for further details.\\

In the present paper we study embeddings of type 
\begin{equation}
\id_\Gamma: \Bsge \rightarrow \Btgz
\label{i-0}
\end{equation}
in some detail. But before we concentrate on these trace spaces we first inspect and slightly improve the related embedding results for spaces on $\rn$. In Theorem~\ref{prop-Bsi-emb} we can show that for two admissible sequences $\bm{\sigma}$ and $\bm{\tau}$, $0<p_1, p_2 < \infty$, $0<q_1,q_2,\leq\infty$, and $q^\ast$ given by $\frac{1}{q^\ast}=\max\left(\frac{1}{q_2}-\frac{1}{q_1},0\right)$,
\[
\id_{\rn} : B^{\bm{\sigma}}_{p_1,q_1}(\rn) \rightarrow B^{\bm{\tau}}_{p_2,q_2}(\rn)
\]
exists and is bounded if, and only if, $p_1\leq p_2$ and $\left(\sigma_j^{-1}\tau_j 2^{jn(\frac{1}{p_1}-\frac{1}{p_2})}\right)_j \in \ell_{q^\ast}$. Next we study embeddings similar to \eqref{i-0}, but where the target space is a Lebesgue space $L_r(\Gamma)$. The corresponding results are of the following type: Assume that $\Gamma$ is an $h$-set satisfying some additional conditions, such that the corresponding trace spaces exist, $\bm{\sigma}$ is an admissible sequence, $0<p\leq r<\infty$ and $0<q\leq\min (r,1)$. Then $\id_\Gamma: \Bsg \rightarrow \ L_r(\Gamma)$ exists and is bounded if, and only if,
$\left(\sigma_j^{-1} h(2^{-j})^{\frac1r-\frac1p}\right)_j \in \ell_{\infty}$, cf. Proposition~\ref{coro-1}. There are further similar characterisations referring to different parameter settings. When $r=p$, the outcome reads as $\id_\Gamma: \Bsg \rightarrow \ L_p(\Gamma)$ exists and is bounded if, and only if,
\begin{equation}
\bm{\sigma}^{-1} \in \begin{cases}
\ell_{q'}, &\text{if}\quad 1\leq p<\infty \quad\text{or}\quad 0<q\leq p<1,\\
\ell_{v_p}, &\text{if}\quad  0<p< 1 \quad\text{and}\quad  0< p < q<\infty,
\end{cases}
\label{0.2 1/2}
\end{equation}
where the number $q'$ is given by $\frac{1}{q'}=\max\left(1-\frac1q,0\right)$ and $v_p$ is given by $\frac{1}{v_p} = \frac1p-\frac1q$; cf. Corollary~\ref{coro-1a}. Actually, such a Corollary shows that, under mild assumptions on the $h$-set $\Gamma$, \eqref{0.2 1/2} gives the exact conditions on $\sigma$, $p$ and $q$ for which the trace spaces $\Bsg$ exist.
%

Our other main results can be found in Theorems~\ref{emb-BGamma-suff} and ~\ref{emb-BGamma-nec}, where we prove necessary and sufficient conditions for the embedding \eqref{i-0} to hold. Apart from some more technical assumptions, ensuring, in particular, the existence of the related trace spaces, the main requirement for the continuity of \eqref{i-0} reads as 
\begin{equation}
\left(\sigma_j^{-1} \tau_j h(2^{-j})^{-\max\left(\frac{1}{p_1}-\frac{1}{p_2},0\right)}\right)_j \in \ell_{q^\ast}.
\label{i-1}
\end{equation}
This condition clearly reflects the expected interplay between smoothness and regularity parameters of the spaces, as well as the underlying geometry of the $h$-set $\Gamma$ where the traces are taken. Moreover, we  can show that an assumption like \eqref{i-1} is also necessary for the embedding \eqref{i-0}, at least when $p_1\leq p_2$.

The paper is organised as follows. In Section~\ref{sect-1} we collect
some fundamentals about $h$-sets and admissible sequences. In
Section~\ref{sect-2} we start by recalling the definition of Besov spaces of generalised
smoothness and their wavelet and atomic characterisations. Then we describe the corresponding trace spaces in some detail and
present the results on growth envelopes that we shall use
afterwards. Our main embedding results are contained in
Section~\ref{sect-3}, first concerning spaces of generalised
smoothness on $\rn$, then for spaces defined on $\Gamma$. Throughout the paper we add remarks, discussions and examples to illustrate the (sometimes technically involved) arguments and results.

\section{Preliminaries}
\label{sect-1}
%
\subsection{General notation}
%
%

As usual, $\rn$ denotes the $n$-dimensional real Euclidean space,
$\nat$ the collection of all natural numbers and $\nat_0=\nat\cup
\{0\}$. We use the equivalence `$\sim$' in
$$
a_k \sim b_k \quad \mbox{or} \quad \varphi(x) \sim \psi(x)
$$
always to mean that there are two positive numbers $c_1$ and $c_2$
such that
$$
c_1\,a_k \leq b_k \leq c_2\, a_k \quad \mbox{or} \quad
c_1\,\varphi(x) \leq \psi(x) \leq c_2\,\varphi(x)
$$
for all admitted values of the discrete variable $k$ or the
continuous variable $x$, where $(a_k)_k$, $(b_k)_k$ are
non-negative sequences and $\varphi$, $\psi$ are non-negative
functions. 
Given two quasi-Banach spaces $X$ and $Y$, we write $X
\hookrightarrow Y$ if $X\subset Y$ and the natural embedding of
$X$ into $Y$ is continuous.\\
All unimportant positive constants will be denoted by $c$,
occasionally with additional subscripts within the same formula.
If not otherwise indicated, $\log$ is always taken with respect to
base 2. For some $\varkappa\in\real$ let 
\begin{equation*}
\varkappa_+ = \max(\varkappa , 0)\quad\mbox{and}\quad \whole{\varkappa} =
\max\{ k\in\mathbb{Z} : k\leq \varkappa\} \, .   
\label{a+}
\end{equation*}
Moreover, for $0<r\leq\infty$
the number $r'$ is given by $\, \frac{1}{r'}:=\left(1-\frac1r\right)_+ $ . 

For convenience,      
let both $\ \dint x\ $ and $\ |\cdot|\ $ stand for the
($n$-dimensional) Lebesgue measure in the sequel. 
For $x\in\rn$ and $r>0$, let $B(x,r)$ be the closed ball
\begin{equation*}
B(x,r) = \left\{y\in\rn : |y-x|\leq r\right\}.
\end{equation*} 

Let $\zn$ stand for the lattice of all points in $\rn$
with integer-valued components, $Q_{j m}$ denote a cube in
$\rn$ with sides parallel to the axes of coordinates, centred at
$2^{-j}m=(2^{-j}m_1,\dots,2^{-j}m_n)$, and with side length
$2^{-j}$, where $m\in \zn$ and $j \in \no$.
If $Q$ is a cube in $\rn$ and $r>0$, then $rQ$ is the cube in $\rn$
concentric with $Q$ and with side length $r$ times the side length
of $Q$. For $0<p<\infty$, $j \in \no$, and $m\in\zn$ we denote by $\chi_{j, m}^{(p)}$ 
the $p$-normalised characteristic function of  the cube $Q_{j, m}$,
\begin{align*}
 \chi_{j, m}^{(p)}(x)= 2^{\frac{j n}{p}} \ \chi_{j, m}(x)= 
\begin{cases}
     2^{\frac{j n}{p}} & \text{ for } \hspace{0.5cm} x\in Q_{j, m}, \\
     0 &  \text{ for }\hspace{0.5cm} x\notin Q_{j, m} ,\\
\end{cases}
\end{align*}
hence $\| \chi_{j, m}^{(p)}|L_p\|=1$. \\

Let $C(\rn)$ be the space of all complex-valued bounded uniformly
continuous functions on $\rn$, equipped with the $\sup$-norm as usual.

\subsection{$h$-sets $\Gamma$}

A central concept for us is the theory of so-called $h$-sets and corresponding
measures. We refer to \cite{Rogers} for a comprehensive treatment of this concept. 

\begin{definition}\quad~ \\[-3ex]
\bli
\item[{\upshape\bfseries (i)\hfill}]
Let $\HH$ denote the class of all positive continuous and non-decreasing
functions $h : (0,1]\to\real$ {\em (gauge functions)} with $\lim_{r\to 0} h(r)=0$. 
\item[{\upshape\bfseries (ii)\hfill}]
Let $h\in\HH$. A non-empty compact set $\Gamma\subset\rn$ is called
{\em $h$-set} if there exists a finite Radon measure $\mu$ with
\begin{align}
\supp\mu & =\Gamma,   \label{supp-h}   \\
\mu(B(\gamma,r))& \sim h(r),\qquad r\in(0,1],\ \gamma\in\Gamma.
\label{hset}
\end{align}
If for a given $h\in\HH$ there exists an $h$-set $\Gamma\subset\rn$, we call
    $h$ a {\em measure function (in~\rn)} and any related measure $\mu$
    with \eqref{supp-h} and~\eqref{hset} will be called {\em $h$-measure (related to $\Gamma$)}.
\eli
\label{defi-hset}
\end{definition}

Certainly one of the most prominent examples of these sets are
the famous $d$-sets, see also Example~\ref{ex-h-fnt} below, but it is also
well-known that in many cases more general approaches are necessary,
cf. \cite[p.~60]{mattila}. Here we essentially follow the presentation in
\cite{bricchi-diss,bricchi-fsdona,bricchi-3,bricchi-4}, see also
\cite{mattila} for basic notions and concepts. For convenience we quote
some results on $h$-sets and give examples afterwards; we refer to the
literature for a more detailed account on geometric properties of $h$-sets. \\

In view of (ii) the question arises which $h\in\HH$ are measure functions. The
complete characterisation is given in \cite{bricchi-fsdona}.

\begin{proposition}     \label{h-measfnt}
Let $h\in\HH$.
There is a compact set $\Gamma$ and a Radon measure $\mu$ with \eqref{supp-h}
and \eqref{hset}
if, and only if, there are constants $0<c_1\leq c_2<\infty$ and a function $\tilde{h} \in \HH$ such that
\[
c_1 \tilde{h}(t) \leq h(t) \leq c_2 \tilde{h}(t),\quad t\in (0,1],
\]
and
\begin{equation*}    \label{main2a}
\tilde{h} (2^{-j}) \leq 2^{kn} \tilde{h}(2^{-k-j}), \qquad \text{for all}\quad j,k\in\no.
\end{equation*}
\end{proposition}                                                                       %

\begin{remark}\label{rem-doubling}
As a consequence of the above proposition we have that for a measure function $h\in\HH$ there is some $c>0$ such that for all $j,k\in\no$,
\begin{equation}    \label{main2}
\frac{h(2^{-k-j})}{h(2^{-j})}\ \geq \ c\ 2^{-kn}.
\end{equation}
Note that every $h$-set $\Gamma$ satisfies the {\em doubling condition}, i.e., 
there is some $c>0$ such that
\begin{equation}
\mu(B(\gamma,2r))  \ \leq \ c\ \mu(B(\gamma,r)),\qquad r\in(0,1],\ \gamma\in\Gamma.
\label{doubling}
\end{equation}
Obviously one can regard \eqref{main2} as a refined version of \eqref{doubling}
for the function~$h$, in which the 
dimension~$n$ of the underlying space \rn~ is taken into account (as expected). 
\end{remark}

\begin{proposition} \label{T:h-sets}
Let $\Gamma$ be  an $h$-set in \rn. 
All $h$-measures
$\mu$ related to $\Gamma$ are equivalent to ${\mathcal H}^h|\Gamma$,
where the latter stands for the restriction to~$\Gamma$ of the generalised
Hausdorff measure with respect to the gauge function~$h$.
\end{proposition}

\begin{remark}\label{rem-dim}
A proof of this result is given in \cite[Thm.~1.7.6]{bricchi-diss}. 
Concerning the theory of generalised 
Hausdorff measures ${\mathcal H}^h$ we refer to \cite[Ch.~2]{Rogers} and
\cite[p.~60]{mattila}; in particular, 
if $h(r)=r^d$, then ${\mathcal H}^h$ coincides 
with the usual $d$-dimensional Hausdorff measure.
\end{remark}

\begin{example}\label{ex-h-fnt}
We restrict ourselves to a few examples only and refer to
\cite[Ex.~3.8]{bricchi-3} for further results. All functions
are defined for $r\in (0,\varepsilon)$, suitably extended on $(0,1]$ afterwards.\\
Let $\Psi$ be a continuous {\em admissible} function or a continuous {\em
  slowly varying} function, respectively. An
{\em admissible} function $\Psi$ in the sense of \cite{ET6}, \cite{moura-diss} is a
positive monotone function on $(0,1]$ such that $\Psi\left(2^{-2j}\right) \sim
\Psi\left(2^{-j}\right)$, $j\in\nat$. A positive and measurable function
  $\Psi$ defined on the interval $(0,1]$ is said to be {\em slowly varying}
  (in Karamata's sense) if
\begin{equation*} \label{a.f.equiv}
\lim_{t\rightarrow 0}\frac{\Psi(s t)}{\Psi(t)}=1,\quad s\in(0,1].
\end{equation*}
For such functions it is known, for instance, that for each $\varepsilon>0$ there is a
decreasing function $\phi$ and an increasing function $\varphi$
with $t^{-\varepsilon}\,\Psi(t)\sim \phi(t)$, and
$t^{\varepsilon}\,\Psi(t)\sim \varphi(t)$; 
we refer to the monograph \cite{BGT} for details and further properties; see also
\cite[Ch.~V]{zygmund}, \cite{EKP}, and \cite{neves-01,neves-02}.
In particular,
\begin{equation}
\Psi_b(x)=\left(1+|\log x|\right)^b, \quad x\in (0,1],\quad
b\in\real, \label{psi-ex-log}
\end{equation}
may be considered as a prototype both for an admissible function and a slowly
varying function. \\
Let $0<d<n$. Then 
\begin{equation}
\label{hPsi}
h(r)=r^d\ \Psi(r),\quad r\in (0,1],
\end{equation}
is a typical example of $h\in\HH$. The limiting cases $d=0$ and $d=n$ can be
included, assuming additional properties of $\Psi$ in view of  \eqref{main2}
and $h(r)\to 0$ for $r\to 0$, e.g.
\beq
h(r)=(1+|\log r|)^b, \quad b<0,\quad r\in (0,1],\label{h-log-ex}
\eeq
referring to \eqref{psi-ex-log}. Such functions $h$ given by \eqref{hPsi} are related
to so-called $(d,\Psi)$-sets studied in \cite{ET6}, \cite{moura-diss}, whereas
the special setting $\Psi\equiv 1$ leads to 
\begin{equation}
\label{h-d}
h(r)=r^d,\quad r\in (0,1],\quad 0< d< n,
\end{equation}
connected with the famous $d$-sets. Apart from \eqref{hPsi} also functions of
type $h(r)=\exp\left(b |\log r|^\varkappa\right)$, $b<0$,
$0<\varkappa<1$, are admitted, for example.
\end{example}

We shall need another feature of some $h$-sets, namely the so-called `porosity'
condition, see also \cite[p.~156]{mattila} and
\cite[Sects.~9.16-9.19]{T-func}.

\begin{definition}
A Borel set $\Gamma\neq\emptyset$ satisfies the {\em porosity condition} if
there exists a number $0<\eta<1$ such that for any ball $B(\gamma,r)$
centred at $\gamma\in\Gamma$ and with radius $0<r\leq 1$ there is a ball
$B(x,\eta r)$ centred at some $x\in\rn$ satisfying
\begin{equation}
B(\gamma,r) \ \supset \ B(x,\eta r), \quad B(x,\eta r) \ \cap \
\overline{\Gamma} \ = \ \emptyset.
\label{ball-1}
\end{equation}
\end{definition}

Replacing, if necessary, $\eta$ by $\frac{\eta}{2}$, we can complement \eqref{ball-1} by
\begin{equation*}
\mathrm{dist}\ \left(B(x,\eta r), \overline{\Gamma}\right) \ \geq \ \eta r,
\quad 0<r\leq 1.
\label{ball-2}
\end{equation*}

This definition coincides with \cite[Def.~18.10]{T-Frac}. There is a complete
characterisation of measure functions $h$ such that the corresponding
$h$-sets $\Gamma$ satisfy the porosity condition; this can be found in
\cite[Prop.~9.18]{T-func}. We recall it for convenience.

\begin{proposition}
Let $\Gamma\subset\rn$ be an $h$-set. Then $\Gamma$ satisfies the porosity
condition if, and only if, there exist constants $c>0$ and $\varepsilon>0$
such that 
\begin{equation}
\frac{h\left(2^{-j-k}\right)}{h\left(2^{-j}\right)} \ \geq \ c\
2^{-(n-\varepsilon)k}, \quad j,k\in\no.
\label{ball-3}
\end{equation}
\label{crit-poros}
\end{proposition}

Note that an $h$-set $\Gamma$ satisfying the porosity condition has Lebesgue
measure $|\Gamma|=0$, but the converse is not true. This can be seen from
\eqref{ball-3} and the result \cite[Prop.~1.153]{T-F3}, 
\begin{equation*}
|\Gamma| = 0 \quad\mbox{if, and only if,}\quad \lim_{r\to 0} \ \frac{r^n}{h(r)}
 = 0
\end{equation*}
for all $h$-sets $\Gamma$.

\begin{remark}\label{example-ball}
In view of our above examples and \eqref{ball-3} it is obvious that $h$ from
\eqref{hPsi} and \eqref{h-d} with $d=n$ does not satisfy the porosity condition, unlike in
case of $d<n$.
\end{remark}

\begin{remark}\label{strongly-isotropic}
Later it will turn out that some additional property of $h$ is desirable; it is in some
sense `converse' to the porosity criterion \eqref{ball-3}: {Triebel} calls the measure $\mu$ corresponding to $h$ and $\Gamma$ with $|\Gamma|=0$ {\em strongly
  isotropic} if $h\in\HH$ is strictly increasing (with $h(1)=1$) and
there exists some $k\in\nat$ such that, for all $j\in\no$,
\beq
2 h\left(2^{-j-k}\right) \leq \ h\left(2^{-j}\right).
\label{cond-strong-iso}
\eeq
It is known that $\mu$ is strongly isotropic if, and only if,
\begin{equation}
\sum_{j=l}^\infty h\left(2^{-j}\right) \sim h\left(2^{-l}\right)
\label{s-i_1}
\end{equation}
for all $l\in\no$, and this is equivalent to 
\begin{equation}
\sum_{j=0}^m h\left(2^{-j}\right)^{-1} \sim h\left(2^{-m}\right)^{-1}
\label{s-i_2}
\end{equation}
for $m\in\no$, see \cite[Def.~7.18, Prop.~7.20]{T-F3}. Plainly, $h(r)=r^d$,
$0<d<n$, satisfies this condition, unlike, e.g., $h(r)=(1+|\log r|)^b$,
$b<0$, from \eqref{h-log-ex}.
\end{remark}

\subsection{Admissible sequences and regularity indices}
\label{adm_seq}

We collect some basic concepts and notions for later use when we
introduce function spaces of generalised smoothness and their trace
spaces on an $h$-set $\Gamma$.

\begin{definition}\label{defi-sigma-adm}
A sequence $\bm{\sigma} = (\sigma_j)_{j\in\no}$ of positive numbers is called {\em
  admissible} if there are two positive constants $\ d_0$, $d_1$ such that
\begin{equation}
d_0 \ \sigma_j \ \leq \ \sigma_{j+1}\ \leq \ d_1\ \sigma_j, \quad j\in\no.
\label{adm-seq}
\end{equation}
\end{definition}
\begin{remark}\label{R-adm-1}
If $\bm{\sigma}$ and $\bm{\tau}$ are admissible sequences, then $\bm{\sigma
  \tau} := \left( \sigma_j \tau_j\right)_j\ $ and $\bm{\sigma}^r :=
\left(\sigma_j^r\right)_j\ $, $r\in\real$, are admissible, too. For convenience, let us further introduce the notation 
\begin{equation*}
\paren{a} := \left(2^{ja}\right)_{j\in\no}\quad \mbox{where}\quad a\in\real.
\end{equation*}
Obviously, for $a,b\in\real$, $r>0$, and $\bm{\sigma}$ admissible,
$\paren{a} \paren{b} = \paren{a+b}$,
$\paren{\textstyle\frac{a}{r}}=\paren{a}^{1/r}$, and $\paren{a}\bm{\sigma}
= \left(2^{ja}\sigma_j \right)_{j\in\no}$.
\end{remark}

\begin{example}\label{ex-sigma}
We restrict ourselves to the sequence $ \bm{\sigma} =
\left(2^{js} \Psi\left(2^{-j}\right)\right)_j $, $ s\in\real$, $\Psi $ an
admissible function in the sense of Example~\ref{ex-h-fnt} above. This includes,
in particular, $ \bm{\sigma} = \paren{s}$, $ s\in\real$. We refer to
\cite{F-L} for a more general approach and further examples.
\end{example}

We introduce some `regularity' indices for $\bm{\sigma}$. 

\begin{definition}
Let $\bm{\sigma}$ be an admissible sequence, and
\begin{equation}
\label{s_sigma}
\lowind{\sigma} := \liminf_{j\to\infty} \
\log\left(\frac{\sigma_{j+1}}{\sigma_j}\right) 
\end{equation}
and
\begin{equation}
\label{s^sigma}
\upind{\sigma} := \limsup_{j\to\infty} \
\log\left(\frac{\sigma_{j+1}}{\sigma_j}\right). 
\end{equation}
\end{definition}

\begin{remark}\label{R-Boyd}
These indices were introduced and used in \cite{bricchi-diss}. For admissible sequences $\bm{\sigma}$
according to \eqref{adm-seq} we have $\ \log d_0 \leq 
 \lowind{\sigma} \leq \upind{\sigma} \leq \log d_1$. One easily
verifies that
\begin{equation}
 \upind{\sigma} = \lowind{\sigma} = s\quad \mbox{in case of}\quad  \bm{\sigma} =
 \left(2^{js} \Psi\left(2^{-j}\right)\right)_j 
\label{index-ex}
\end{equation}
for all admissible functions $\Psi $ and $s\in\real$. In contrast 
to this one can find examples in \cite{F-L}, due to
{Kalyabin}, showing that an admissible sequence has not
necessarily a fixed main order. Moreover, it is known that 
for any $0<a\leq b<\infty$ there is
an admissible sequence $\bm{\sigma}$ with
$\lowind{\sigma} =a$ and $\upind{\sigma} =b$,
that is, with prescribed upper and lower indices.

It is more or less obvious from the definitions
\eqref{s_sigma}, \eqref{s^sigma} that for admissible sequences
$\bm{\sigma}, \bm{\tau}$ one obtains  $\lowind{\sigma} = -
\sol\left(\bm{\sigma}^{-1}\right)$, $\sul\left(\bm{\sigma}^r\right) = r\lowind{\sigma}$, $r \geq 0$, 
and
$\upind{\sigma \tau} \leq \upind{\sigma} + \upind{\tau}$,
$\lowind{\sigma \tau} \geq \lowind{\sigma} + \lowind{\tau}$.
In particular, for $\bm{\sigma} = \paren{a}$, $a\in\real$, this can be sharpened by
$\sol\left(\bm{\tau} \paren{a}\right)   = a + \upind{\tau}$, $\sul\left(\bm{\tau} \paren{a}\right) = a + \lowind{\tau}$. 
Observe that,
given $\varepsilon>0$, there are two positive constants
$c_1=c_1(\varepsilon)$ and $c_2=c_2(\varepsilon)$ such that
\begin{equation} \label{estimate}
c_1 \ 2^{(\lowind{\sigma}-\varepsilon)j} \leq \sigma_j \leq c_2 \
2^{(\upind{\sigma}+\varepsilon)j},\qquad j\in\nat_0.
\end{equation}
Plainly this implies that whenever $\lowind{\sigma} >0$, then
$\bm{\sigma}^{-1}$ belongs to any space $\ell_u$, $0<u\leq\infty$, whereas
$\upind{\sigma}<0 $ leads to $\bm{\sigma}^{-1}\not\in\ell_\infty$.
\end{remark}

\begin{remark}\label{R-boyd-2}
Note that in some later papers, cf. \cite{bricchi-4}, instead of \eqref{s_sigma} and
  \eqref{s^sigma} the so-called {\em upper} and {\em lower Boyd
  indices of} $\bm{\sigma}$ are considered, given by 
\begin{equation*}
\alpha_{\bm{\sigma}} = \lim_{j\to\infty} \ \frac1j \ 
\log \left(\sup_{k\in\no} \ \frac{\sigma_{j+k}}{\sigma_k}\right) \ = \ \inf_{j\in\nat} \ \frac1j \ 
\log \left(\sup_{k\in\no} \ \frac{\sigma_{j+k}}{\sigma_k}\right)
\end{equation*}
and
\begin{equation*}
\beta_{\bm{\sigma}} = \lim_{j\to\infty} \ \frac1j \ 
\log \left(\inf_{k\in\no} \ \frac{\sigma_{j+k}}{\sigma_k}\right) \ = \ \sup_{j\in\nat} \ \frac1j \ 
\log \left(\inf_{k\in\no} \ \frac{\sigma_{j+k}}{\sigma_k}\right)
\end{equation*}
respectively. In general we have 
$$
\lowind{\sigma}  \leq \beta_{\bm{\sigma}} \leq  \alpha_{\bm{\sigma}}\leq
 \upind{\sigma},
$$
but one can construct admissible sequences with $\lowind{\sigma}  <
\beta_{\bm{\sigma}}$ and $\alpha_{\bm{\sigma}} <  \upind{\sigma}$. 
\end{remark}

\begin{remark}\label{subseq-adm}
We briefly mention some convenient feature of admissible sequences that will be used later several times. Assume that $\bm{\gamma}$ is some admissible sequence with the property that for some (fixed) $\iota_0\in\nat$ and some $u\in (0,\infty]$ the special subsequence $(\gamma_{k\iota_0})_k\subset\bm{\gamma}$ 
satisfies $(\gamma_{k\iota_0})_k \in \ell_u$. This already implies $\bm{\gamma}\in\ell_u$, since, for arbitrary $m\in\nat$,
\[
\sum_{j=1}^{m\iota_0} \gamma_j^u = \sum_{k=1}^m \gamma_{k\iota_0}^u + \sum_{l=1}^{\iota_0-1} \sum_{k=0}^{m-1} \gamma_{l+k\iota_0}^u \sim \sum_{k=1}^m \gamma_{k\iota_0}^u
\]
with constants depending on $\bm{\gamma}$, $u$ and $\iota_0$, but not on $m\in\nat$. Thus $\bm{\gamma}\in\ell_u$ if, and only if, $(\gamma_{k\iota_0})_k \in \ell_u$. The modifications for $u=\infty$ are obvious. 
\end{remark}

We shall need the following short lemma several times below and insert
it here. It mainly relies on an inequality of Landau, see
\cite[Thm.~161, p.~120]{hardy}, see also \cite[Prop.~4.1]{Ca-L-Lloc}.

\begin{lemma}
\label{lemma-Landau}
Let $\bm{\alpha} = (\alpha_k)_{k\in\no} $, $\bm{\beta} =
(\beta_k)_{k\in\no}$ be sequences of non-negative numbers, let
$0<q_1,q_2\leq\infty$, and $q^\ast$ be given by 
\begin{equation}
\frac{1}{q^\ast} \ := \ \left(\frac{1}{q_2} - \frac{1}{q_1}\right)_+\ . 
\label{q-ast}
\end{equation}
Then
\[
\bm{\alpha} \in \ell_{q^\ast} \qquad\text{if, and only if,}\qquad \bm{\alpha} \bm{\beta} \in \ell_{q_2} \quad \text{whenever}\quad \bm{\beta}\in\ell_{q_1}.
\]
\end{lemma}

\begin{proof}
The `only if\,'-part is a consequence of H\"older's inequality (when
$q^\ast<\infty$) or the monotonicity of the $\ell_u$-scale (when
$q^\ast=\infty$). Moreover, the `if\,'-part in case of $q^\ast<\infty$
is covered by the above-mentioned result of Landau, see
\cite[Thm.~161, p.~120]{hardy}. So we are left to show the necessity
of $\bm{\alpha}\in\ell_\infty$ when assuming  that $\bm{\alpha}\bm{\beta} \in \ell_{q_2}$ whenever $\bm{\beta}\in\ell_{q_1}$, where in particular we have now that $q_1\leq q_2$. This is surely well-known, but for the sake of
completeness we insert a short argument here. The case $q_1=\infty$
being obvious, we need to deal with the case $q_1<\infty$ merely. \\
Assume, to the contrary, that $\bm{\alpha}$ is unbounded. Then we can find a subsequence $(\alpha_{k_j})_j\subset \bm{\alpha}$ which is strictly increasing and satisfies $\alpha_{k_j} \xrightarrow[j\to\infty]{} \infty$. We refine this subsequence further by the following procedure: let $\alpha_{t_1}= \alpha_{k_1}$ and choose consecutively the index $t_{j+1}$, $j\in\nat$, such that 
\[
\frac{\alpha_{t_{j+1}-1}}{\alpha_{t_j}} < 2\quad\text{and} \quad \frac{\alpha_{t_{j+1}}}{\alpha_{t_j}} \geq 2,\quad j\in\nat,
\]
that is, $t_{j+1}$ is the smallest index larger than $t_j$ such that $\alpha_{t_{j+1}}\geq 2\alpha_{t_j}$. Now consider the sequence $\bm{\beta }=(\beta_k)_k$ given by
\begin{align*}
\beta_k & = \begin{cases} j^{\frac{1}{q_1}} \alpha_{t_j}^{-1}, & k=t_j, \ j\in\nat,\\ 0, & \text{otherwise}.\end{cases}
\intertext{Then}
\|\bm{\beta} | \ell_{q_1}\| &= \Big(\sum_{j=1}^\infty j \alpha_{t_j}^{-q_1}\Big)^{\frac{1}{q_1}} \ = \ \Big(\sum_{j=1}^\infty j \Big(\frac{\alpha_{t_j}}{\alpha_{t_{j-1}}} \cdots \frac{\alpha_{t_2}}{\alpha_{t_1}} \Big)^{-q_1} \alpha_{t_1}^{-q_1}\Big)^{\frac{1}{q_1}} \\
&\leq \ \alpha_{k_1}^{-1} \Big(\sum_{j=1}^\infty j 2^{-(j-1)q_1}\Big)^{\frac{1}{q_1}} \ < \infty,
\intertext{whereas}
\|\bm{\alpha}\bm{\beta} | \ell_{q_2}\| &= \Big(\sum_{j=1}^\infty \alpha_{t_j}^{q_2} j^{\frac{q_2}{q_1}} \alpha_{t_j}^{-q_2} \Big)^{\frac{1}{q_2}} \ = \ \Big(\sum_{j=1}^\infty j^{\frac{q_2}{q_1}} \Big)^{\frac{1}{q_2}} = \infty
\end{align*}
(modification if $q_2=\infty$), which disproves $\bm{\alpha} \bm{\beta} \in \ell_{q_2}$. 
\end{proof}

We return to the notion of strongly isotropic measures given in Remark~\ref{strongly-isotropic} and will from
now on use the following abbreviations introduced in \cite{bricchi-3}: 
Let $\Gamma$ be some $h$-set, $h\in \HH$ a measure function. Then we denote 
\begin{equation}
\bm{h} := \left(h_j\right)_{j\in\no} \quad \mbox{with}\quad  h_j := h(2^{-j}),
\quad j\in\no,
\label{h-seq}
\end{equation}
for the sequence connected with $h\in\HH$. 

\begin{remark}\label{strong-iso-index}
It is easy to verify that the hypothesis $\upind{\bm{h}}<0$ implies
that there exists $\kappa_0>1$ such that $h_{j+1} \leq \kappa_0^{-1}
h_j$ for $j\geq j_0$ and some appropriately chosen $j_0 \in
\no$. Therefore $\bm{h} \in \ell_1$ and we can also say that the measure $\mu$ corresponding to $h$ is strongly
isotropic (at least if we relax, for convenience, the assumptions $h(1)=1$ and $|\Gamma|=0$ and
also admit that \eqref{cond-strong-iso} holds for all $j\in\nat$
with $j\geq j_0$ for some fixed starting term $j_0\in\nat$ only), cf. Remark~\ref{strongly-isotropic}. More precisely, for
arbitrary admissible sequences $\bm{\sigma}$ the assumption
$\upind{\bm{\sigma}}<0$ leads to $\bm{\sigma}\in\ell_1$ and also to the equivalences \eqref{s-i_1},
\eqref{s-i_2} with $\bm{h}$ replaced by $\bm{\sigma}$ .
\end{remark}

\section{Besov spaces}
\label{sect-2}

The main object of the paper are embeddings of Besov spaces defined (as
trace spaces) on some $h$-set $\Gamma$. We approach this concept now.

\subsection{Besov spaces of generalised smoothness}
%
\label{sect-2-0}

First we want to introduce function spaces of generalised smoothness and need
to recall some notation. By $\mathcal{S}(\rn)$ we denote the Schwartz
space of all complex-valued, infinitely differentiable and rapidly
decreasing functions on $\rn$ and by $\mathcal{S}'(\rn)$ the dual space
of all tempered distributions on $\rn$. If $\varphi \in
\mathcal{S}(\rn)$, then
\begin{equation}\label{Ftransform}
\widehat{\varphi}(\xi)\equiv (\mathcal F
\varphi)(\xi):=(2\pi)^{-n/2}\int_{\rn}e^{-ix\xi}\varphi(x) \dint
x, \quad \xi\in \rn,
\end{equation}
denotes the Fourier transform of $\varphi$. As usual, ${\mathcal
F}^{-1} \varphi$ or $\varphi^{\vee}$ stands for the inverse
Fourier transform, given by the right-hand side of
\eqref{Ftransform} with $i$ in place of $-i$. Here $x\xi$ denotes
the scalar product in $\rn$. Both $\mathcal F$ and ${\mathcal
F}^{-1}$ are extended to $\mathcal{S}'(\rn)$ in the standard way. Let
$\varphi_0\in\mathcal{S}(\rn)$ be such that
\begin{equation*}  \label{phi}
\varphi_0(x)=1 \quad \mbox{if}\quad |x|\leq 1 \quad \mbox{and}
\quad \supp \varphi_0 \subset \{x\in\rn: |x|\leq 2\},
\end{equation*}
and for each $j\in\nat$ let
\begin{equation*}\label{phi-j}
\varphi_j(x):=\varphi_0(2^{-j}x)-\varphi_0(2^{-j+1}x), \quad x\in
\rn.
\end{equation*}
Then the sequence $(\varphi_j)_{j=0}^{\infty}$ forms a smooth dyadic resolution of
unity.

\begin{definition}
Let $\bm{\sigma}$ be an admissible sequence, $0<p,q\leq\infty$, and
$(\varphi_j)_{j=0}^{\infty}$ a smooth dyadic resolution of
unity as described above. Then
\begin{equation*}
\Bsi(\rn) = \left\{ f\in \mathcal{S}'(\rn) ~: \biggl( \sum_{j=0}^\infty
\sigma_j^q \left\| {\mathcal F}^{-1} \varphi_j {\mathcal F}f |
L_p(\rn)\right\|^q \biggr)^{1/q} < \infty\right\}
\end{equation*}
$($with the usual modification if $q=\infty)$.
\end{definition}
\begin{remark}\label{R-Bsigma}
These spaces are quasi-Banach spaces, independent of the chosen resolution of
unity, and $\mathcal{S}(\rn)$ is dense in $\Bsi(\rn)$ when $p<\infty$ and
$q<\infty$. Taking $\bm{\sigma} = \paren{s}$, $ 
s\in\real$, we obtain the classical Besov spaces $\B(\rn)$, whereas $ \bm{\sigma}
= (2^{js} \Psi(2^{-j}))_j $, $ s\in\real$, $\Psi $ an
admissible function, leads to spaces $B^{(s,\Psi)}_{p,q}(\rn)$, studied in
\cite{moura, moura-diss} in detail. Moreover, the above spaces $\Bsi(\rn)$
are special cases of the more general approach investigated in \cite{F-L}. For
the theory of spaces $\B(\rn)$ we refer to the series of monographs 
\cite{T-F1,T-F2,T-Frac,T-func,T-F3}.
\end{remark}

For later use, we briefly describe the wavelet and atomic characterisation
of  Besov spaces with generalised smoothness obtained in
\cite{almeida} and \cite{F-L}, respectively. Let $\widetilde{\phi}$ be a scaling function 
on $\real$ with compact support and of sufficiently high regularity.
Let $\widetilde{\psi}$ be an associated compactly supported wavelet. Then the 
tensor-product ansatz yields
a scaling function $\phi$  and associated wavelets
$\psi_1, \ldots, \psi_{2^{n}-1}$, all defined now on $\rn$. 
We suppose $\widetilde{\phi} \in C^{N_1}(\real)$ and $\supp
\widetilde{\phi}, \supp \widetilde{\psi}
\subset [-N_2,\, N_2]$ for certain natural numbers $N_1$ and $N_2$.
This implies
\begin{equation}\label{2-1-2}
\phi, \, \psi_l \in C^{N_1}(\rn) \quad \text{and} \quad 
\supp \phi ,\, \supp \psi_l \subset [-N_3,\, N_3]^n , 
\end{equation}
for $l=1, \ldots \, , 2^{n}-1$. We use the standard abbreviations 
\begin{equation*}\label{convention}
\phi_{j,m}(x) =  \phi(2^j x-m) \quad
\text{and}\quad
\psi^l_{j,m}(x) =   \psi_l(2^j x-m) . 
\end{equation*}
To formulate the result we  introduce some suitable sequence spaces. 
For $0<p<\infty$, $0<q \leq\infty$, and $\bm{\sigma}$ an admissible sequence, let
\begin{multline*}
b^{\bm{\sigma}}_{p,q} :=   \Bigg\{ \lambda = 
\{\lambda_{j,m}\}_{j,m} : \     \lambda_{j,m} \in \comp\, , 
\\
\| \, \lambda \, |b^{\bm{\sigma}}_{p,q}\| = \Big\| 
\Big\{ \sigma_j 2^{- j \frac{n}{p} }\,  \Big\|\sum_{m \in
  \zn}\lambda_{j,m}\, \chi^{(p)}_{j,m}| L_p(\rn) 
\Big\|\Big\}_{j\in\no} | \ell_q\Big\| < \infty \Bigg\}. 
\end{multline*}
In view of the definition of $Q_{j,m}$ and $\chi^{(p)}_{j,m}$ one can easily verify that  
\begin{align*}\label{bspr}
\big\|\lambda |b^{\bm{\sigma}}_{p,q}\big\| & \sim  
\bigg\| \Big\{ \sigma_j  2^{- j \frac{n}{p} } \Big(\sum_{m \in
  \zn}|\lambda_{j,m}|^p \    \Big)^{1/p}\Big\}_{j\in\no} | \ell_q \bigg\|.
\end{align*}

\begin{theorem}\label{wavesigma}  
Let $0 < p<\infty$, $0<q \le \infty$ and  $\bm{\sigma}$ be an admissible
sequence. Then there exists a number $N_0=N_0(\bm{\sigma},p,n)$ and
for any $N_1>N_0$ a
scaling function $\phi$  and wavelets $\psi^l$,  $l=1, \ldots ,2^n-1$,
as above satisfying \eqref{2-1-2}, such that the
following holds: A distribution $f \in \SpRn$ belongs to $\Bsi(\rn)$
if, and only if, it can be represented by
\[ f= \sum_{m\in\zn} \mu_m  \phi_{0,m} +
\sum_{l=1}^{2^n-1} \sum_{j=0}^\infty \sum_{m \in
  \zn} \lambda^l_{j,m}\psi^l_{j,m}\quad \text{with}\quad
\mu\in\ell_p\quad\text{and}\quad \lambda^l\in b^{\bm{\sigma}}_{p,q},
\]
$l=1, \dots, 2^n-1$ (unconditional convergence in $\SpRn$). Moreover, the coefficients can
be uniquely determined by
\[
\mu_m = 2^{jn} \langle f,\phi_{0,m}\rangle\quad\text{and}\quad
\lambda^l_{j,m}=2^{jn}\langle f,\psi^l_{j,m}\rangle, \quad m\in\zn,
\ j\in\no, \ l=1, \dots, 2^n-1,
\]
and 
\begin{align*}
\| \, f \, |\Bsi(\rn)\|^\star  = \ &  
\Big\| \left\{\langle f,\phi_{0,m}\rangle \right\}_{m\in \zn} |
\ell_p\Big\| 
+ \sum_{l=1}^{2^n-1}
\Big\| \left\{\langle f,\psi^l_{j,m}\rangle \right\}_{j\in \no, m\in \zn} | b^{\widetilde{\bm{\sigma}}}_{p,q}\Big\|
\end{align*}
is an 
equivalent $($quasi-$)$ norm in $\Bsi(\rn)$, where $\widetilde{\bm{\sigma}} =
\bm{\sigma}\paren{n}$. 
\end{theorem} 

\begin{remark}\label{waverem}
It follows from Theorem~\ref{wavesigma} that, under the conditions given, the mapping
$$
T\,:\,f\;\mapsto \; \Big( \left\{\langle f,\phi_{0,m}\rangle \right\}_{m\in \zn}, \left\{\langle
f,\psi^l_{j,m}\rangle \right\}_{j\in \no, m\in \zn, l=1,\ldots, 2^n-1}\Big)
$$
is an isomorphism of $\Bsi(\rn)$ onto $\ell_p\oplus
\big(\oplus_{l=1}^{2^n-1} b^{\widetilde{\bm{\sigma}}}_{p,q}\big)$,
$\widetilde{\bm{\sigma}} = \bm{\sigma}\paren{n}$,
cf. \cite[Cor.~17]{almeida}. Moreover, if $q<\infty$ and $N_1$ is
chosen large enough, then the wavelet system forms an unconditional
Schauder basis in $\Bsi(\rn)$, \cite[Cor.~16]{almeida}.
\end{remark}

{
\begin{definition} \label{atoms-Bsigma}
Let $K \in \nat_0$ and $c> 1$.\\[-3ex]
\bli
\item[{\upshape\bfseries (i)\hfill}]
A $K$ times differentiable
complex-valued function $a(x)$ in $\rn$ $($continuous if $K=0)$ is
called {\em an $1_K$-atom} if
$$
\supp a \subset c\,Q_{0 m}\quad\mbox{for some}\quad m\in\zn,
$$
and
$$
|\Dd^{\alpha}a(x)|\leq 1 \quad \mbox{for} \quad |\alpha|\leq K.
$$
\item[{\upshape\bfseries (ii)\hfill}]
Let $L+1\in\nat_0$, and $\bm{\sigma}$ admissible. A $K$
times differentiable complex-valued function $a(x)$ in $\rn$
$($continuous if $K=0)$ is called {\em an $(\bm{\sigma}, p)_{K,L}$-atom}
if, for some $j \in \no$,
$$
\supp a \subset c\,Q_{j m}\quad\mbox{for some}\quad m\in \zn,
$$
$$
|\Dd^{\alpha}a(x)|\leq \sigma_j^{-1} \ 2^{j(\frac{n}{p}+|\alpha|)}\, \quad
\mbox{for} \quad |\alpha|\leq K,\quad x\in\rn,
$$
and
$$
\int\limits_{\rn} x^{\beta}a(x) \dint x=0 \quad \mbox{if}
\quad|\beta|\leq L.
$$
\eli
\end{definition}

We adopt the usual convention to denote atoms located at $Q_{j m}$ as
above by
$a_{j,m}$, $j\in\no$, $m\in\zn$. Let $b_{p,q} =
b_{p,q}^{\paren{n/p}}$, $0<p<\infty$, $0<q\leq\infty$.
The atomic decomposition theorem for $\Bsi(\rn)$ reads as follows, see
\cite[Thm.~4.4.3, Rem. 4.4.8]{F-L}.

\begin{proposition} \label{atomicdecomposition}
Let $\bm{\sigma}$ be admissible, $0<p,q\leq\infty$, $c>1$, $K\in \no$ and $L+1\in\no$ with
\begin{equation*}
K > \upind{\sigma} \qquad\mbox{and}\qquad L > -1 +
n\left(\frac{1}{\min(1,p)} -1\right) - \lowind{\sigma}
\label{moment-Bsi}
\end{equation*}
be fixed. Then $f\in {\mathcal S}'(\rn)$ belongs to
$\Bsi(\rn)$ if, and only if, it can be represented as
\begin{equation}\label{series}
f=\sum_{j=0}^{\infty} \sum_{m\in\zn} \lambda_{j, m}\,a_{j,m}(x), 
\quad\mbox{convergence being in } {\mathcal S}'(\rn),
\end{equation}
where $\lambda \in b_{p,q}$ and  $a_{j, m}(x)$ are $1_K$-atoms $(j=0)$ or
$(\bm{\sigma},p)_{K,L}$-atoms $(j \in \nat)$ according to 
Definition~\ref{atoms-Bsigma}. Furthermore
$$
\inf \|\lambda\,|\,b_{p,q}\|,
$$
where the infimum is taken over all admissible representations
\eqref{series}, is an equivalent quasi-norm in $\Bsi(\rn)$.
\end{proposition}
}

\subsection{Trace spaces}
%
\label{sect-2-1}
Let $\Gamma$ be some $h$-set, $h\in \HH$, and recall the notation
\eqref{h-seq} for the sequence connected with $h\in\HH$. 
Given $0<p \leq \infty$, and similarly as in $L_p(\rn)$, but now with respect to the measure $\mu \sim {\mathcal H}^h|\Gamma$
related to the $h$-set $\Gamma$, $L_p(\Gamma) = L_p(\Gamma, \mu)$ stands for the quasi-Banach space of $p$-integrable (measurable, essentially bounded if $p =\infty$) functions on $\Gamma$ with respect to the measure $\mu$, quasi-normed in the obvious way. Suppose there exists
some $c>0$ such that, for all $\varphi\in\mathcal{S}(\rn)$,
\begin{equation}
\label{trace-def}
\left\| \varphi\raisebox{-0.4ex}[1.1ex][0.9ex]{$|_{\Gamma}$}~ | L_p(\Gamma)
\right\| \leq \ c\ \left\| \varphi | \Btau(\rn)\right\|,
\end{equation}
where the trace on $\Gamma$ is taken pointwise. By the density of
$\mathcal{S}(\rn)$ in $\Btau(\rn)$ for $p,q<\infty$ and the completeness of
$L_p(\Gamma)$  one can thus define for $f\in \Btau(\rn)$ in those cases its
trace $\tr{\Gamma} f = f\raisebox{-0.4ex}[1.1ex][0.9ex]{$|_{\Gamma}$}$ on
$\Gamma$ by completion of pointwise restrictions. 

\begin{remark}\label{R-unique}
In this way, $\tr{\Gamma}: \Btau(\rn) \to L_p(\Gamma)$ will be
uniquely determined by the fact that it is linear, continuous and
coincides with the pointwise restriction when restricted to
$\SRn$. The definition of $\Bsg$ as the trace space of $\Bsih$ in
$L_p(\Gamma)$ means to take $\Bsg=\tr{\Gamma}\Btau(\rn)$ for
$\bm{\tau}=\bm{\sigma} \bm{h}^{1/p} \paren{n}^{1/p}$ for which the
above approach works.
\end{remark}

The starting point for us was the trace result of {Bricchi} \cite[Thm.~5.9]{bricchi-3}, 
\begin{equation} \label{trace-Lp}
\tr{\Gamma} B^{\bm{h}^{1/p}\paren{n}^{1/p}}_{p,q}\!(\rn) = L_p(\Gamma),
\end{equation}
where $0<p<\infty$, $0<q\leq\min(p,1)$, and $\Gamma$ is an $h$-set satisfying the porosity condition
\eqref{ball-1} (this condition not being required in the proof of the embedding $\hookrightarrow$) . 
Based on this result {Bricchi} proposes in
\cite[Def.~3.3.5]{bricchi-diss} the approach to define spaces $\Bsg$
as trace spaces, conjugating general embedding results for Besov
spaces on $\rn$ (as seen for example in \cite[Thm.~3.7]{C-Fa} or
\cite[Prop.~2.2.16]{bricchi-diss}) with the fact that
\eqref{trace-def} above holds for $\bm{\tau}=\bm{h}^{1/p}\paren{n}^{1/p}$, $0<p<\infty$, $0<q\leq \min (p,1)$ (cf. \cite[Thm.~3.3.1(i)]{bricchi-diss}. Then we have, following \cite[Def.~3.3.5]{bricchi-diss}, that for $0<p,q<\infty$ and $\bm{\sigma}$ admissible with  $\lowind{\sigma}
> 0$,  Besov spaces $\Bsg$ on $\Gamma$ are defined as

\begin{equation} 
\label{defBesovGamma}
\Bsg := \tr{\Gamma} \Bsih,
\end{equation}
more precisely,
\begin{equation} 
\label{precisely}
\Bsg := \left\{ f\in L_p(\Gamma) : \exists \ g\in \Bsih,\tr{\Gamma} g =   
    f\right\},
\end{equation}
equipped with the quasi-norm
\begin{equation}
\label{qnormBG}
\left\| f |  \Bsg\right\| = \inf\left\{ \left\| g | \Bsih\right\| :  
   \tr{\Gamma} g =  f,  g\in \Bsih\right\}.
\end{equation}

This was extended in the following way in \cite[Def.~3.7]{Cae-env-h}:

\begin{definition}
Let $0<p,q<\infty$, $\bm{\sigma}$ be an admissible sequence, and $\Gamma$ be an
$h$-set. Assume that \\[-3ex]
\bli
\item[{\upshape\bfseries (i)\hfill}]
in case of $p\geq 1$ or $q\leq p<1$, 
\begin{equation*}
\label{ext-2-4-a}
\bm{\sigma}^{-1} \in \ell_{q'},
\end{equation*}
\item[{\upshape\bfseries (ii)\hfill}]
in case of $0<p<1$ and $p<q$,
\begin{equation*}
\bm{\sigma}^{-1} \bm{h}^{\frac1r-\frac1p}\in \ell_{v_r} \quad \text{for some}\quad  
r\in [p,\min(q,1)] \quad \mbox{and}\quad
\frac{1}{v_r} = \frac1r-\frac1q, 
\label{ext-2-4-b}
\end{equation*}
\eli
is satisfied. Then (it makes sense to) define $\Bsg$  as in
\eqref{defBesovGamma}, \eqref{precisely}, that is, as the trace space
of $\Bsih$ in $L_p(\Gamma)$,
again with the  quasi-norm given by \eqref{qnormBG}.
\label{defi-Bsg}
\end{definition}

\begin{remark}\label{reasonable}
Note that the -- in \eqref{defBesovGamma} implicitly given --
correspondence of smoothness $\paren{1}$  (that is, 0, in classical
notation) on $\Gamma$ and smoothness $\bm{h}^{1/p}\paren{n}^{1/p}$ on
$\rn$ is in good agreement with \eqref{trace-Lp}. 
\end{remark}

\begin{remark}\label{infty-cases}
For simplicity we restrict ourselves to the case
  $p<\infty$ and $q<\infty$ in the above definition and also in our results below, though both in
  \cite{bricchi-diss} and in \cite{Cae-env-h} the definition of Besov
  spaces on $\Gamma$ also covers the cases when $p$ or $q$ can be
  $\infty$. Then the above approach has to be modified
  appropriately. Moreover, we refer to \cite{Cae-env-h} (and also
  \cite{CaH-3}) for some argument to what extent our definition extends
  the former approach by {Bricchi}.
\end{remark}

\begin{remark}\label{BDef-d-sets}
For $\ h(r)=r^d$, $0<d<n$,  $\Gamma$ a $d$-set, and $\bm{\sigma} = \paren{s}$,
$s\in\real$, the above definition \eqref{defBesovGamma} can be rewritten
as 
\begin{equation*} 
\mathbb{B}^s_{p,q}(\Gamma) = \tr{\Gamma} B^{s+\frac{n-d}{p}}_{p,q}(\rn),
\end{equation*}
assuming $0<p,q<\infty$, $s>0$, and
\begin{equation*} 
\mathbb{B}^0_{p,q}(\Gamma) = \tr{\Gamma} B^{\frac{n-d}{p}}_{p,q}(\rn) = L_p(\Gamma),
\end{equation*}
if $0<p<\infty$, $0<q\leq \min(p,1)$. This coincides with
\cite[Def.~20.2]{T-Frac}. There is a parallel result for
$(d,\Psi)$-sets $\Gamma$ studied by {Moura} in
\cite{moura,moura-diss}, which yields for $0<p,q<\infty$, $s>0$,
and admissible $\Psi$,
\begin{equation*} 
\mathbb{B}^{(s,\Psi)}_{p,q}(\Gamma) = \tr{\Gamma}
B^{(s+\frac{n-d}{p},\Psi^{1+\frac1p})}_{p,q}(\rn), 
\end{equation*}
see \cite[Def.~2.2.7]{moura-diss}.
\end{remark}

In Section~\ref{sect-3-2} we briefly return to our above approach
\eqref{trace-def} and ask what happens if we replace the target space
$L_p$ by some possibly larger or smaller space $L_r$.


\subsection{Growth envelopes} 

\label{ge-d}
We want to apply some result on growth envelopes obtained in
\cite{Cae-env-h} later and thus briefly recall this concept and the
result for spaces $\Bsg$. Therefore we concentrate on this specific setting
mainly and refer for further details to \cite{Ha-crc} as well as
\cite{Cae-env-h}. 

Let $\Gamma$ be an $h$-set according to Definition~\ref{defi-hset}
with the corresponding Radon measure $\mu\sim{\mathcal H}^h|\Gamma$ and 
 $f $ be a $\mu-$measurable function on $\Gamma$, finite
$\mu-$a.e. Its decreasing (i.e., non-increasing) rearrangement $\ f^{\ast,\mu}\ $ is the function 
$\ f^{\ast,\mu}\ $ defined on $[0,\infty)$ by
\begin{equation*}
f^{\ast,\mu}(t) = \inf\left\{ s\geq 0 : \mu\left(\{ \gamma\in \Gamma :
|f(\gamma)|>s\}\right) \leq t\right\}\quad,\quad t\geq 0\;.
\label{f*}
\end{equation*}
We put $\ \inf\emptyset = \infty$, as usual. Note that $\ f^{\ast,\mu}\ $ is
non-negative, decreasing and right-continuous on $[0,\infty)$, 
and $f$ and $f^{\ast,\mu}$ are equi-measurable.
Furthermore, $\ f^{\ast,\mu}(0) = \|f| L_\infty(\Gamma)\|$, and $\
f^{\ast,\mu}(t)=0$ for $t > \mu(\Gamma)$.

\begin{remark}\label{omega-nu}
There is plenty of literature
on the topic of non-increasing rearrangements in general measure spaces; we refer to \cite[Ch.~2, Prop.~1.7]{BS} and
\cite[Ch.~2, \S 2]{dVL}, for instance. We decided to indicate the
measure $\mu$ here since when dealing with functions on $\rn$ and their
trace $\tr{\Gamma} f = f\raisebox{-0.4ex}[1.2ex][0.9ex]{$|_{\Gamma}$}$ on a
set $\Gamma\subset\rn$, typically with $|\Gamma|=0$, then it
naturally matters which measure $\nu$ (on $\rn$ or $\Gamma$) is taken for the decreasing
rearrangements $f^{\ast,\nu}$. But we want to avoid any further discussion here and
stick to the above situation.
\end{remark}
%
\begin{definition}\label{defi-eg}
Let $[\Omega,\nu]$ be some measure space and $X$ a quasi-normed space of
 $\nu$-measurable functions on $\Omega$, finite a.e. in $\Omega$.
A non-negative function $\egX$ defined on some interval
$(0,\varepsilon]$, $\varepsilon\in (0,1)$, is called the {\em
 $($local$)$ growth envelope function} of $X$ if 
\begin{equation*}
\egX(t) \sim \sup_{\|f|X\| \leq 1} \, f^{\ast,\nu}(t)\;,\quad t\in (0,\varepsilon].
\label{eq-eg}
\end{equation*} 
\end{definition}

\begin{remark}\label{eg-remark}
For a more general
approach, appropriate interpretations as well as a detailed account on basic properties of $\egX$ we refer to
\cite{HaHabil} and \cite{Ha-crc}. Note that there exist function spaces $X$ which do not
possess a growth envelope function in the sense that $\egX$ is not finite for
any $t>0$. It is well-known that $  X \hookrightarrow L_\infty $ if, and only if,
$ \egX$ is bounded. Let us finally mention the convenient `monotonicity' feature: 
\begin{equation}
X_1 \hookrightarrow X_2\quad\mbox{implies}\quad \egv{X_1}(t) \leq c\ \egv{X_2}(t)\quad\mbox{
for some $c>0$ and all $t\in (0,\varepsilon)$,}
\label{eg-XX}
\end{equation}
where $\varepsilon=\varepsilon(X_1,X_2)>0$ has to be chosen sufficiently
small.
\end{remark}

We refine the characterisation of $ X$ provided by the growth envelope
function and introduce some `characteristic' 
index $ \uGindex{X}$, which gives a finer measure of the (local) integrability
of functions belonging to $ X$. 
By $ \mu_{\mathsf G} $ we mean the Borel measure associated with the
non-decreasing function $ - \log \egX$, where the growth envelope
function $\egX$ of $X$ ~is assumed to be positive, non-increasing and continuous on some interval $(0,\varepsilon]$, with $\varepsilon > 0$, and $ X\not\hookrightarrow L_\infty$. 
This approach essentially coincides with the one
presented by {Triebel} in \cite[Sect.~12.1]{T-func} and
\cite[Sect.~12.8]{T-func}, see also \cite{Ha-crc} and \cite{Cae-env-h} for further details.

\begin{definition}
Let $[\Omega,\nu]$ be some measure space and $\ X
\not\hookrightarrow L_\infty \ $ some quasi-normed space of
$\nu$-measurable functions, finite a.e. on $\Omega$. Let $ \egX $ be a
positive, non-increasing and continuous growth envelope function defined on $(0,\ve]$
for some sufficiently small $\ve>0$. The index $ \uGindex{X}$, 
$ 0<\uGindex{X} \leq\infty$, is defined as the infimum of all
numbers $ v$, $\ 0<v\leq \infty$, such that
\begin{equation}
\left(\int\limits_0^\varepsilon \left[ \frac{f^\ast(t)}{\egX(t)}\right]^v
\mu_{\mathsf G}(\dint t)\right)^{1/v} \leq \; c\; \| f|X\|
\label{eq-envg}
\end{equation}
(meaning $\  
\sup_{0<t<\varepsilon} \ \frac{f^\ast(t)}{\egX(t)} \leq \; c\; \|
f|X\|$ ~ when $\ v=\infty$) 
holds for some $\ c>0\ $
and all $f\in X$. Then 
\begin{equation*}
\envg\big( X \big) = \left(\egX(\cdot), \uGindex{X} \right)
\label{eq-growth-env}
\end{equation*}
is called the {\em (local) growth envelope} for the function space $X$.
\label{defi-envg}
\end{definition}

\begin{remark}\label{post-defi-envg}
Since \eqref{eq-envg} holds with $ v=\infty $ in any case and the corresponding expressions on the
left-hand
side are, up to multiplicative (positive) constants, non-increasing in $ v $ by
\cite[Prop.~12.2]{T-func}, it is reasonable to ask for the {\em smallest} parameter $v$
satisfying \eqref{eq-envg}. Note that parallel to \eqref{eg-XX} there
is a similar `monotonicity' feature for
corresponding indices that can be found in \cite{Ha-crc}.
\end{remark}


\begin{example}\label{env-Lpq}
Let $\Gamma\subset\rn$ be 
an $h$-set in the sense of Definition~\ref{defi-hset} and
$0<p<\infty$.  Then 
$$
\envg\left(L_p(\Gamma)\right) \ = \left(t^{-\frac1p}, \ p\right).
$$
This result can be found in \cite[Prop.~2.11]{Cae-env-h} (as a special
case of a corresponding result for the more general Lorentz spaces $L_{p,q}(\Gamma)$). In case of $d$-sets, $h(r)=r^d$, $0<d< n$, the result remains
unchanged and coincides with the usual $\rn$-situation as well as more
abstract settings, cf. \cite[Sect.~4]{Ha-crc}.
\end{example}

Before we can present our results for spaces $\Bsg$ we need some
preparation. We shall use a convenient notation introduced in
\cite{C-DM-2} and adapted 
in \cite{C-Fa}. A function $\Lambda : (0,\infty) \to (0,\infty)$
is called {\em admissible} if it is continuous and satisfies for any $b>0$
that $ \Lambda(y) \sim \Lambda(by)$, $y>0$, with equivalence constants
independent of $y$. Let $\bm{\tau}$ be an admissible sequence according to
Definition~\ref{defi-sigma-adm}, and $\bm{N} = \left(N_j\right)_j\ $ a sequence of
positive numbers such that for some $\lambda_0>1$ it holds $\lambda_0 N_j \leq
N_{j+1}$, $j\in\no$. In \cite[Ex.~2.3]{C-Fa} a construction
of an admissible function $\Lambda=\Lambda_{\bm{\tau},{\bm{N}}}$ is given, that satisfies
\begin{equation}
\Lambda_{\bm{\tau},{\bm{N}}}(y) \ \sim \ \tau_j, \quad y \in \left[N_j,
  N_{j+1}\right], \quad j\in \no.
\label{Lambda}
\end{equation}

Let $\bm{\sigma}$ be an admissible sequence and, for $h\in \HH$, let $N_j = h_j^{-1}$, $j\in\no$, satisfy $\lambda_0 N_j \leq
N_{j+1}$, $j\in\no$, for some $\lambda_0>1$. We denote
$\Sigma_{\bm{\sigma},\bm{h}} := \Lambda_{\bm{\sigma},\bm{h}^{-1}}$. It satisfies, in particular,
\begin{equation}
\Sisih (y) \ \sim \ \sigma_j, \quad y \in
  \left[h_j^{-1}, h_{j+1}^{-1}\right], \quad j\in \no.
\label{Lambda-sh}
\end{equation}

Let $\ J_0\in\nat$ be chosen such that $h_{J_0} <
1$. For $0<r,u\leq\infty$ and $\Sisih$ we introduce the function $\Psi_{r,u} :
\left(0,h_{J_0}\right] \rightarrow \real$ by
\begin{equation}
\Psi_{r,u}(t) := \left(\ \int\limits_t^1 \ y^{-\frac{u}{r}}\
\Sisih(y^{-1})^{-u}\ \frac{\dint y}{y}\right)^{1/u},
\label{Psi-ru}
\end{equation}
with the usual modification for $u=\infty$,
\begin{equation}
\Psi_{r,\infty}(t) := \sup_{t\leq y\leq 1}\ y^{-\frac{1}{r}}\
\Sisih(y^{-1})^{-1} .
\label{Psi-ru'}
\end{equation}
In particular, $\Psi_{r,u}$ is positive, monotonically
decreasing and continuous, being also differentiable when $u\neq\infty$ --- cf.
\cite[Lemma~2.5]{C-Fa}. 
The main result on growth envelopes in Besov spaces $\Bsg$ was obtained in \cite{Cae-env-h}.


\begin{theorem}
\label{below 2}
Let $0<p,q<\infty$, $\bm{\sigma}$ be admissible, $\Gamma$ be an
$h$-set satisfying the porosity condition. Assume 
\begin{equation}
-n \leq \lowind{h} \leq \upind{h} < 0, 
\label{T-1}
\end{equation}
\begin{equation}
\bm{\sigma}^{-1} \bm{h}^{-\frac1p} \not\in \ell_{q'} 
\label{T-2}
\end{equation}
and
\begin{equation}
\lowind{\sigma} > - \lowind{h} \left(\frac1p-1\right)_+\ .
\label{T-3}
\end{equation}
Let $\Psi_{r,u}$ be given by 
\eqref{Psi-ru}. Then
\begin{equation}
\envg \Bsg = (\Psi_{p,q'},q).
\label{gepair}
\end{equation}
\end{theorem}

\begin{remark}\label{less hyp}
Observe that the hypotheses that $\Gamma$ satisfies the porosity condition,
$\lowind{\bm{h}} \geq -n$  and
\eqref{T-3} are not needed for the proof of the optimality of the exponent
$q$, as long as we assume that $\Bsg$ exists as the trace of $\Bsih$ in
$L_p(\Gamma)$. Moreover, (\ref{T-2}) ensures that $\Bsg$ is not contained in $L_\infty(\Gamma)$;
see also (\ref{B-C-1}) below.
\end{remark}

\begin{remark}\label{GEgeq}
In some applications it will be enough to use a sharp estimate from
below for the growth envelope function, in which case we can also
dispense with most of the hypotheses considered in the Theorem above,
as follows from \cite[Prop.~4.4]{Cae-env-h}: up to a multiplicative constant factor, $\Psi_{p,q'}$ from \eqref{gepair} is a lower bound for the local growth envelope function of $\Bsg$ near 0 as long as this space is assumed to exist as the trace of $\Bsih$ in $L_p(\Gamma)$ and the $h$-set $\Gamma$ satisfies the condition $\upind{\bm{h}} < 0$.
\end{remark}

\begin{example}\label{env-d-set}
If $\Gamma$ is a $d$-set, $h(r)=r^d$, $0<d< n$, and $0<p,q<\infty$,
$d\left(\frac1p-1\right)_+<s<\frac{d}{p}$, then
Theorem~\ref{below 2} reads as
\begin{align}
\envg \mathbb{B}^{s}_{p,q}(\Gamma) & = \left(t^{\frac{s}{d}-\frac1p}, q\right).\label{envg-d-1}
\intertext{In the limiting case $s=\frac{d}{p}$, now with
  $1<q<\infty$, we obtain}
\envg \mathbb{B}^{s}_{p,q}(\Gamma) & = \left(|\log t|^{1/q'}, q\right).\nonumber
\end{align}
\end{example}

\begin{remark}\label{R-Lp-Gamma}
Note that in this last example in case \eqref{envg-d-1} the growth envelope
function $\egX(t)\sim t^{\frac{s}{d}-\frac1p}$ does not distinguish between spaces
$\mathbb{B}^{s_1}_{p_1,q_1}(\Gamma)$ and
$\mathbb{B}^{s_2}_{p_2,q_2}(\Gamma)$ whenever
$s_1-\frac{d}{p_1}=s_2-\frac{d}{p_2}$, not even in case
of $s_1=s_2$, $p_1=p_2$, but $q_1\neq q_2$. Therefore the introduction of the
additional fine index is justified.
\end{remark}

\begin{remark}\label{R-env-d-sets}
We would like to emphasise the resemblance of the above results in
Example~\ref{env-d-set} with
those obtained for spaces $\B(\rn)$, see \cite[Thms.~13.2,
  15.2]{T-func}: apparently one just had to replace $n$ by $d$ to get the right
expressions in the context of $d$-sets. But, as already remarked in
\cite{Cae-env-h}, this obvious similarity has indeed some geometrical meaning, as it is
now easy to see that, for all possible relations between parameters given above, and for small positive $t$,
\beq
\egv{\mathbb{B}^{s}_{p,q}(\Gamma)}\left(t^d\right) \sim \egv{B^{s+\frac{n-d}{p}}_{p,q}(\rn)}\left(t^n\right),
\label{geom-env}
\eeq
with coincidence of the corresponding indices. In particular, the
argument on the left-hand side of \eqref{geom-env} represents the geometry of the $d$-set
$\Gamma$ on which the trace space is defined, whereas the argument on the
right-hand side corresponds to the geometry of the underlying $\rn$
for the space $B^{s+(n-d)/p}_{p,q}(\rn)$ from which the trace (on $\Gamma$)
is taken.
\end{remark}

\begin{example}\label{envg-dPsi}{For the convenience of the reader we briefly return to
  the study of $(d,\Psi)$-sets $\Gamma$ and corresponding trace spaces
  $\mathbb{B}^{(s,\Psi)}_{p,q}(\Gamma)$, recall
  Remark~\ref{BDef-d-sets}. In this case the assumptions in
  Theorem~\ref{below 2} are satisfied whenever $0<d<n$, $0<p,q<\infty$,
  $s>d(\frac1p-1)_+$, $\Psi$ is admissible and
\[
\left(2^{j(\frac{d}{p}-s)}
  \Psi(2^{-j})^{-\frac1p-1}\right)_j\notin\ell_{q'},
\]
that is, when
\[ \begin{cases} s < \frac{d}{p} &\quad  \text{or}
    \\ s=\frac{d}{p} & \text{and}\ 
    \left(\Psi(2^{-j})^{-\frac1p-1}\right)_j\notin\ell_{q'}.\end{cases}
\]
This simply follows from the setting $h_j \sim
2^{-jd}\Psi\left(2^{-j}\right)$ and
$\sigma_j=2^{js}\Psi\left(2^{-j}\right)$, $j\in\nat$, in this special
case of $(d,\Psi)$-sets and spaces of type
${B}^{(s,\Psi)}_{p,q}$. This implies $\Sisih
\left(2^{jd}\Psi(2^{-j})^{-1}\right)\sim 2^{js}\Psi(2^{-j}) $, in view of (\ref{Lambda-sh}), and straightforward calculations taking advantage of the admissibility of $\Psi$ --- cf. also Remark \ref{R-env-d-Psi-sets} below --- lead to
\beq
\envg \mathbb{B}^{(s,\Psi)}_{p,q}(\Gamma) = \left(t^{\frac{s}{d}-\frac1p}\Psi(t)^{-1-\frac{s}{d}}, q\right)
\label{eg-subcrit-psi}
\eeq
if $s < \frac{d}{p}$, and to 
{
\begin{align*}
\envg \mathbb{B}^{(s,\Psi)}_{p,q}(\Gamma)  = 
\left(\left(\int_{t^{1/d}}^1 \Psi(y)^{-q'\left(\frac1p+1\right)} \frac{\dint
  y}{y}\right)^{{1/q'}}, q\right)
\end{align*}
{if} $ s =
\frac{d}{p}$ and
$\ \left(\Psi(2^{-j})^{-\frac1p-1}\right)_j\notin\ell_{q'}$. In case
of $0<q\leq 1$, that is, $q'=\infty$, this last expression has to be understood as 
\begin{align}
\envg \mathbb{B}^{(s,\Psi)}_{p,q}(\Gamma)  = 
\left(\sup_{t^{1/d}\leq y \leq 1} \Psi(y)^{-\left(\frac1p+1\right)} , q\right),
\label{gecaseless1}
\end{align}
assuming now that
$\ \left(\Psi(2^{-j})^{-\frac1p-1}\right)_j\notin\ell_{\infty}$ and $ s =
\frac{d}{p}$. Actually, with this assumption, and recalling the definition of $\Psi$, the growth envelope function in \eqref{gecaseless1} can simply be written as $\Psi(t)^{-\left(\frac1p+1\right)}$.
}
}\end{example}

\begin{remark}\label{R-env-d-Psi-sets}
In the spirit of Remark~\ref{R-env-d-sets} we would like to compare the
outcome in Example~\ref{envg-dPsi} with its counterpart for growth
envelopes of spaces ${B}^{(s,\Psi)}_{p,q}(\rn)$ obtained in
\cite{C-DM,C-DM-2}. In the sub-critical case $n(\frac1p-1)_+<s<\frac{n}{p}$ that
result reads as
\[
\envg\left({B}^{(s,\Psi)}_{p,q}(\rn)\right) =
\left(t^{\frac{s}{n}-\frac1p} \Psi(t)^{-1},q\right),
\]
whereas in the limiting case $s=\frac{n}{p}$ with the additional
assumption 
$\left(\Psi(2^{-j})^{-1}\right)_j\not\in\ell_{q'}$ it is
\[
\envg\left({B}^{(s,\Psi)}_{p,q}(\rn)\right) =
\left( \tilde{\Phi}_{p,q'}(t),q\right)\qquad\text{with}\quad
\tilde{\Phi}_{p,q'}(t) = \left(\int_{t^{1/n}}^1 \Psi(y)^{-q'} \frac{\dint y}{y}\right)^{1/q'}
\]
(modified similarly as in \eqref{gecaseless1} when $q'=\infty$). We claim that
\beq
\egv{\mathbb{B}^{(s,\Psi)}_{p,q}(\Gamma)}\left(t^d\Psi(t)\right) \sim \egv{B^{(s+\frac{n-d}{p},\Psi^{1+\frac1p})}_{p,q}(\rn)}\left(t^n\right),
\label{geom-env-Psi}
\eeq
parallel to \eqref{geom-env} in case of $d$-sets (i.e., with
$\Psi\equiv 1$). 

We first deal with the sub-limiting case. Plainly, \eqref{eg-subcrit-psi} leads to
\begin{align*}
\egv{\mathbb{B}^{(s,\Psi)}_{p,q}(\Gamma)}\left(t^d\Psi(t)\right) &
\sim \left(t^d\Psi(t)\right)^{\frac{s}{d}-\frac1p}
\Psi\left(t^d\Psi(t)\right)^{-\frac{s}{d}-1} \ 
 \sim \ t^{s-\frac{d}{p}} \Psi(t)^{-\frac{1}{p}-1}\\[1ex]
& \sim (t^n)^{\frac1n(s+\frac{n-d}{p})-\frac1p}\ 
\Psi\left(t^n\right)^{-\frac{1}{p}-1}\  \sim \ \egv{B^{(s+\frac{n-d}{p},\Psi^{1+\frac1p})}_{p,q}(\rn)}\left(t^n\right),
\end{align*}
where we used twice the admissibility of $\Psi$, which implies 
$\Psi(t)\sim \Psi(t^a \Phi(t))$, $a>0$, $\Phi$ admissible, for small $t>0$. 
As to the critical case, the correspondence between the additional assumptions on $\Psi$ when
$s=\frac{d}{p}$ (on the side of the trace spaces on $\Gamma$) and when
$s=\frac{n}{p}$ (on the side of the corresponding spaces on $\rn$) are obvious, whereas \eqref{geom-env-Psi} follows again from the admissibility of $\Psi$.
\end{remark}

\section{Embeddings}
\label{sect-3}
%
\subsection{Sharp embeddings for spaces on $\rn$}
\label{sect-3-1}
First we study criteria for embeddings of spaces of type $\Bsi(\rn)$ and slightly extend corresponding results in \cite[Thm.~3.7]{C-Fa} with fore-runners in
\cite{moura-diss} and  \cite[Prop.~4.12]{bricchi-3}.

\begin{theorem}
\label{prop-Bsi-emb}
Let $\bm{\sigma}$ and $\bm{\tau}$ be two admissible sequences, $0<p_1, p_2 < \infty$, $0<q_1,q_2\leq\infty$, and $q^\ast$ be given by \eqref{q-ast}.
\bli
\item[{\upshape\bfseries (i)\hfill}]
The embedding
\beq\label{BB-Rn}
\id: B^{\bm{\sigma}}_{p_1,q_1}(\rn) \longrightarrow
B^{\bm{\tau}}_{p_2,q_2}(\rn)
\eeq
exists and is bounded if, and only if,
\begin{equation}
p_1\leq p_2\qquad\text{and}\qquad 
\bm{\sigma}^{-1}\bm{\tau}\paren{n}^{\frac{1}{p_1}-\frac{1}{p_2}}
\in \ell_{q^\ast}\ .  
\label{emb-nec}
\end{equation}
\item[{\upshape\bfseries (ii)\hfill}]
The embedding \eqref{BB-Rn} is never compact.
\eli
\end{theorem}

\begin{proof}
We begin with (i). The sufficiency of 
\eqref{emb-nec} for the existence and boundedness of  the embedding \eqref{BB-Rn} is covered by
\cite[Thm.~3.7]{C-Fa} already. So it remains to show the necessity.
Here we use the wavelet characterisation in Theorem~\ref{wavesigma},
see also Remark~\ref{waverem}, and assume to choose the wavelet system in such a way
that it is applicable for both spaces simultaneously. Hence the existence and continuity of the embedding \eqref{BB-Rn} implies
that $\ell_{p_1}\hookrightarrow \ell_{p_2}$ and 
\beq\label{bb-Rn}
b^{\widetilde{\bm{\sigma}}}_{p_1,q_1}  \hookrightarrow b^{\widetilde{\bm{\tau}}}_{p_2,q_2} .
\eeq
Obviously, the first embedding immediately yields $p_1\leq p_2$. As
for \eqref{bb-Rn}, 
note first that
$\widetilde{\bm{\sigma}}^{-1}\widetilde{\bm{\tau}}\paren{n}^{\frac{1}{p_1}-\frac{1}{p_2}}
= \bm{\sigma}^{-1}\bm{\tau}\paren{n}^{\frac{1}{p_1}-\frac{1}{p_2}}$. 
We consider the sequence $\bm{\lambda} =\{\lambda_{j,m}\}_{j\in\no,m\in\zn}$, defined by 
\[
\lambda_{j,m}=\begin{cases} \sigma_{j}^{-1} 2^{j n(\frac{1}{p_1}-\frac12)} \beta_j, & m=0, \ j\in\no, \\ 0, & \text{otherwise},\end{cases}
\]
where $\bm{\beta}\in\ell_{q_1}$ is arbitrary. Then by straightforward
calculation,
\[
\|\bm{\lambda}| b^{\widetilde{\bm{\sigma}}}_{p_1,q_1}\| =  \left\|\bm{\beta} | \ell_{q_1}\right\|
\qquad \text{and} \qquad 
\|\bm{\lambda}| b^{\widetilde{\bm{\tau}}}_{p_2,q_2}\| = \left\|\bm{\tau} \bm{\sigma}^{-1} \paren{n}^{\frac{1}{p_1}-\frac{1}{p_2}} \bm{\beta} | \ell_{q_2}\right\|.
\]
Application of Lemma~\ref{lemma-Landau} with $\bm{\alpha}=\bm{\tau} \bm{\sigma}^{-1} \paren{n}^{\frac{1}{p_1}-\frac{1}{p_2}}$ leads to the desired result \eqref{emb-nec}.\\

We deal with (ii) and benefit from the fact that embeddings between classical spaces of type $\B(\rn)$ are never compact independently of the choice of the parameters. We proceed by contradiction. Assume that \eqref{BB-Rn} was compact. Choose $s>\upind{\sigma}$ and $t<\lowind{\tau}$, such that $B^{s}_{p_1,q_1}(\rn) \hookrightarrow B^{\bm{\sigma}}_{p_1,q_1}(\rn) $ and $B^{\bm{\tau}}_{p_2,q_2}(\rn)\hookrightarrow B^{t}_{p_2,q_2}(\rn)$ in view of (i) and Remark~\ref{R-Boyd}, in particular \eqref{estimate}. Hence we obtain the compactness of $B^{s}_{p_1,q_1}(\rn) \hookrightarrow B^{t}_{p_2,q_2}(\rn)$ which can never happen. So we have proved (ii).
\end{proof}

\begin{remark}\label{rem-HaSM-2}{
Note that there is an extension 
in \cite{Ha-SM-2} to spaces of type $B^{\bm{\sigma},
  \bm{N}}_{p,q}(\rn)$. 
}\end{remark}

\begin{example}\label{sharp-BPsi}
We return to the setting of $(d,\Psi)$-sets, $0<d<n$,
  $\Psi$ admissible, as described in Example~\ref{ex-h-fnt}, in
  particular in \eqref{hPsi}, and the corresponding spaces
  $B^{(s,\Psi)}_{p,q}(\rn)$ discussed in Remark~\ref{R-Bsigma}, with $ \bm{\sigma}
= (2^{js} \Psi(2^{-j}))_j $, $ s\in\real$, $0<p,q\leq\infty$, studied in
\cite{moura, moura-diss} in detail. Let $s_i\in\real$, $\Psi_i$
admissible, $0<p_i<\infty$, $0<q_i\leq\infty$, $i=1,2$. Using the well-known
abbreviation for the difference of the corresponding differential
smoothnesses,
\[
\delta= s_1-\frac{n}{p_1}-\left(s_2-\frac{n}{p_2}\right),
\]
Theorem~\ref{prop-Bsi-emb}(i) reads as 
\begin{align*}
B^{(s_1,\Psi_1)}_{p_1,q_1}(\rn) \ & \hookrightarrow
B^{(s_2,\Psi_2)}_{p_2,q_2}(\rn) 
\intertext{if, and only if,}
p_1\leq p_2\qquad & \mbox{and}\qquad 
\left(2^{-j\delta}\frac{\Psi_2\left(2^{-j}\right)}{\Psi_1\left(2^{-j}\right)}\right)_j\in\ell_{q\ast}.
\end{align*}
This coincides (in the limiting case $\delta=0$) with our earlier result \cite[Prop.~4.3]{CaH-2} which
already refined an outcome of
{Moura} in \cite[Prop.~1.1.13]{moura-diss}.
\end{example}

For matter of comparison we recall the result for target spaces $L_p$, see
\cite[Cor.~3.18]{C-Fa} and the slightly weaker version
\cite[Prop.~4.13]{bricchi-3}.  Here we rely on the recently obtained
characterisation (in the more general setting of spaces $B^{\bm{\sigma},\bm{N}}_{p,q}(\rn)$) in \cite[Thm.~4.3, Cor.~4.6(i)]{Ca-L-Lloc}.

\begin{proposition}\label{sharp-B-Lloc}
Let $0<p<\infty$, $0<q\leq\infty$, $\bm{\sigma}$ be admissible. Then
\begin{align*}
&\Bsi(\rn)  \ \hookrightarrow \ L_{\max(p,1)}(\rn)
\intertext{if, and   only if,}
&\begin{cases}
\bm{\sigma}^{-1}\paren{n}^{\left(\frac1p-1\right)} \in \ell_{q'},
&\text{if}\quad  
0<p\leq 1, \ 0<q\leq\infty,\\
\bm{\sigma}^{-1}\in\ell_\infty, & \text{if}\quad  1<p<\infty, \ 0<q\leq\min(p,2),\\
\bm{\sigma}^{-1}\in\ell_{\frac{pq}{q-p}}, &\text{if}\quad  1<p\leq 2, \ \min(p,2)<q\leq\infty,\\
\bm{\sigma}^{-1}\in\ell_{\frac{2q}{q-2}}, &\text{if}\quad  2<p<\infty, \ \min(p,2)<q\leq\infty.\end{cases}
\end{align*}
\end{proposition}

We come to the target space $L_\infty(\rn)$ now.

\begin{proposition}
\label{Bsi-C}
Let $\bm{\sigma}$ be an admissible sequence, $0<p,q\leq\infty$. Then
\begin{equation*}
\Bsi(\rn) \hookrightarrow C(\rn) \quad \mbox{if, and only if,}\quad
\bm{\sigma}^{-1} \paren{n}^{\frac1p}  \in \ell_{q'}\ ,
\end{equation*}
where $C(\rn)$ can be replaced by $L_\infty(\rn)$.
\end{proposition}

This characterisation was proved in \cite[Cors.~3.10, 4.9]{C-Fa} in
full generality; it has a fore-runner in \cite{Kal-5}, restricted to
$1<p,q<\infty$. In case of spaces $B^{(s,\Psi)}_{p,q}(\rn)$, that is, when $ \bm{\sigma}
= (2^{js} \Psi(2^{-j}))_j $, $ s\in\real$, the result was already obtained in \cite[Prop.~3.11]{C-DM-2}.

\begin{remark}\label{include-infty}{
The case $p=\infty$ missing in Proposition~\ref{sharp-B-Lloc} can be
read from Proposition~\ref{Bsi-C} with $p=\infty$.
}\end{remark}

\subsection{Sharp embeddings for spaces on $\Gamma$}
\label{sect-3-2}

We now study embedding assertions of type 
$$\Bsg \hookrightarrow
L_r(\Gamma)\qquad \mbox{and}\qquad \Bsge \hookrightarrow \Btgz\ .$$ 
In contrast to previous
considerations in \cite{bricchi-diss}, see also \cite{bricchi-5}, we
are especially interested in `sharp' embeddings, that is, to give
necessary and sufficient conditions for the existence of such
embeddings, whereas in the above-mentioned papers the focus was rather
on the compactness of related embeddings.  \\

First we collect what is 
already known. Note that whenever $ 0<p, q<\infty$, $\bm{\sigma}$ is
an admissible sequence, and $\Gamma$ is an
$h$-set, then
\[   \Bsg \hookrightarrow L_{\max(p,1)}(\Gamma)
\qquad\text{if}\qquad
\bm{\sigma}^{-1}\bm{h}^{-\left(\frac1p-1\right)_+}\in\ell_{q'}\ ,
\]
see \cite[Prop.~3.9]{Cae-env-h}. We can sharpen this assertion below
in Corollary~\ref{C-Bgamma-Lloc}. 
Moreover, in \cite[Prop.~3.11, Cor.~4.5]{Cae-env-h} there is also a complete characterisation of embeddings in $L_\infty(\Gamma)$: 

\begin{corollary}
\label{nec for C}
Let $0<p,q < \infty$ and $\bm{\sigma}$ be an admissible sequence. Let $\Gamma$
be an $h$-set with $\upind{\bm{h}}<0$  and assume there exists the space $\Bsg$ as
the trace of $\Bsih$ in $L_p(\Gamma)$.  Then
\beq\label{B-C-1}
\Bsg \hookrightarrow L_\infty(\Gamma)\quad\text{if, and only if,}\quad \bm{\sigma}^{-1}
\bm{h}^{-\frac{1}{p}} \in \ell_{q'},
\eeq
where $L_\infty(\Gamma)$ can be replaced by $C(\Gamma)$.
\end{corollary}

\begin{remark}\label{R-iff-Linfty}{
We would like to mention that for the `if\,'-part in \eqref{B-C-1} we
do not need to assume a priori the existence of $\Bsg$ (this comes as
a consequence then). For that part we also do not need the assumption $\upind{\bm{h}}<0$.
}\end{remark}

\begin{example}\label{B-C-dsets}
For $d$-sets $\Gamma$ and the special choice $\bm{\sigma} = \paren{s}$ we obtain $\mathbb{B}^s_{p,q}(\Gamma)
\hookrightarrow C(\Gamma)\ $ if, and only if, $s>\frac{d}{p}$, or (in the limiting case) $s=\frac{d}{p}$ and $0<q\leq 1$; this is in
perfect coincidence with the results in \cite[Thm.~20.6, Sect.~21.1]{T-Frac},
and compares well with the classical situation on $\rn$, see
\cite[Thm.~3.3.1(ii)]{SiT}.
\end{example}

Now we return to the problem of `sharp' embeddings into spaces $L_r$
with $r<\infty$. First we deal with the question what happens if we replace the target space
$L_p$ in \eqref{trace-def} by some possibly larger or smaller space
$L_r$. More precisely, we study the boundedness of the trace operator
\beq\label{tr-bdd-r<p}
\tr{\Gamma} : \Bsih \to L_r(\Gamma),\quad f \mapsto \tr{\Gamma} f,
\eeq
defined first pointwise for smooth functions
$\varphi\in\mathcal{S}(\rn)$ and then extended by completion in such
spaces where $\mathcal{S}(\rn)$ is dense. Therefore we require
$\max(p,q)<\infty$ again. We recall that this trace is well-defined if we
assume \eqref{trace-def} to hold (with $\bm{\tau} =\bm{\sigma}
\bm{h}^{1/p} \paren{n}^{1/p}$ and $L_p$ replaced by $L_r$). It will
turn out that, in all cases we shall study, \eqref{tr-bdd-r<p} exists
and is continuous if, and only if, the trace space $\Bsg$ exists and 
\beq\label{emb-Bsg-Lr}
\Bsg \hookrightarrow L_r(\Gamma)
\eeq
holds. We begin with the case $r\leq p$.

\begin{proposition}
Let $\Gamma$ be an $h$-set satisfying the porosity condition, $\bm{\sigma}$ be an
admissible sequence and either

\begin{description}
	\item[(a)] $1 \leq p < \infty$, $\, 0<q<\infty$, $\, 0<r \leq p \;\;$ or
	\item[(b)] $0<q,r \leq p < 1$.
\end{description}

Then
\begin{equation}
\tr{\Gamma}: \Bsih \ \longrightarrow \ L_r(\Gamma)
\label{tracer<p}
\end{equation} 
exists and is bounded if, and only if, 
$$ 
\bm{\sigma}^{-1} \in \ell_{q'}. $$
\label{ex-trace}
\end{proposition}

\begin{proof}
Note that the sufficiency is already covered by our considerations in Section~\ref{sect-2-1},
together with the embeddings $\ L_p(\Gamma) \hookrightarrow L_r(\Gamma)$ for
$0 < r\leq p$ and the compacity of $\Gamma$. It remains to show the necessity. But this comes as an easy consequence of the density results which we have obtained in \cite[Rem.~3.20,Thm.~3.18]{CaH-3}, namely that when $\bm{\sigma}^{-1} \notin \ell_{q'}$, then ${\mathcal D}(\rn \setminus \Gamma)$ is dense in $\Bsih$. For the convenience of the reader, we essentially repeat the argument presented in the beginning of Section 3.2 in \cite{CaH-3}, which shows that, in the presence of the just mentioned density, the trace (\ref{tracer<p}) cannot exist: 

Assume that  $\mathcal{D}(\rn\setminus\Gamma)$ is dense in $\Btau(\rn)$.
Let $\varphi\in C^\infty_0(\rn)$ with
$\varphi\equiv 1$ on a neighbourhood of $\Gamma$. Clearly, $\varphi \in \SRn \cap
\Btau(\rn)$ and therefore there exists a sequence $(\psi_k)_k \subset
\mathcal{D}(\rn\setminus\Gamma) \subset \SRn$ with
$$\left\| \varphi - \psi_k | \Btau(\rn) \right\|
\xrightarrow[k\to\infty]{} 0.$$
If the trace $\tr{\Gamma}: \Btau \ \longrightarrow \ L_r(\Gamma)
$ were to exist, this would imply
$$0 = \psi_k\raisebox{-0.4ex}[1ex][0.9ex]{$|_{\Gamma}$} = \tr{\Gamma}
\psi_k \xrightarrow[k\to\infty]{} \tr{\Gamma} \varphi =
\varphi\raisebox{-0.4ex}[1ex][0.9ex]{$|_{\Gamma}$} = 1 \quad
\text{in}\ L_r(\Gamma),$$
which is a contradiction.
\end{proof}

\begin{remark}\label{R-trace-ex}{
In view of Definition~\ref{defi-Bsg}(i) the above result states that,
under the given conditions, there is a trace in $L_p(\Gamma)$ if, and
only if, there is a trace in $L_r(\Gamma)$. 
Note that {Triebel} proved in \cite[Cor.~7.21]{T-F3} a related
result in case of $1<p<\infty$, $1\leq q<\infty$,
$1\leq r\leq p$, and $\bm{\sigma}\paren{n}^{\frac1p} \bm{h}^{\frac1p} =
\paren{s}$ with $ s>0$: also in that case the trace in $L_r(\Gamma)$
exists if, and only if, $\bm{\sigma}^{-1}\in\ell_{q'}$ (the further
restrictions on the parameters are partly caused by the used argument, in
particular, duality). Moreover, when $q>1$, then he can even show compactness of the trace operator in that setting.
}\end{remark}

We return to \eqref{emb-Bsg-Lr} now in case of $r\geq p$.

\begin{proposition}
\label{coro-1}
Let $\Gamma$ be an $h$-set such that $\upind{\bm{h}}<0$ and $\bm{\sigma}$ be an admissible sequence.\\[-3ex]
\bli
\item[{\upshape\bfseries (i)\hfill}]
Let $\ 0<p\leq r<\infty \,$ and $\, 0<q\leq\min (r,1)$. Then
$$\tr{\Gamma}: \Bsih \ \longrightarrow \ L_r(\Gamma)$$
exists and is bounded if, and only if,
$$\bm{\sigma}^{-1} \bm{h}^{\frac1r-\frac1p} \in \ell_{\infty}.$$ 
\item[{\upshape\bfseries (ii)\hfill}]
Let $\, 0<q<\infty\,$ and $\, 0<p\leq r \leq \min(q,1)$ 
and denote $\frac{1}{v_r}:=\frac1r-\frac1q$. Then
$$\tr{\Gamma}: \Bsih \ \longrightarrow \ L_r(\Gamma)$$
exists and is bounded if, and only if,
$$\bm{\sigma}^{-1} \bm{h}^{\frac1r-\frac1p} \in \ell_{v_r}.$$
\eli
\end{proposition}

\begin{proof} 
{\em Step 1}.\quad We show the sufficiency of the conditions by applying Theorem~\ref{prop-Bsi-emb} to get
$$\Bsih \hookrightarrow B^{\bm{h}^{1/r}\paren{n}^{1/r}}_{r,\min( r,1)}\!(\rn),
$$
so that, with the help of \eqref{trace-Lp}, we see that the required traces exist.\\

{\em Step 2}.\quad We come to the necessity of the condition in (i). 
Assuming the existence of the trace in assertion (i), since $p\leq r$ and $\Gamma$ is compact we have that $L_r(\Gamma) \hookrightarrow L_p(\Gamma)$ and therefore $\Bsg$ exists as the trace of $\Bsih$ in $L_p(\Gamma)$ and, moreover, 
\begin{equation}
\Bsg \hookrightarrow L_r(\Gamma).
\label{eq:trace}
\end{equation}
Then Remark~\ref{GEgeq} can be applied to get that
\begin{equation}
\egv{\Bsg}(h_j) \gtrsim \sup_{k=0,\ldots,j} \sigma_k^{-1}h_k^{-\frac{1}{p}} \geq \sigma_j^{-1}h_j^{-\frac{1}{p}} \quad \mbox{for $\, j\,$ large enough}.
\label{eq:lower}
\end{equation}
On the other hand, from \eqref{eq:trace} and \eqref{eg-XX} we can write
$$\egv{\Bsg}(t) \lesssim \egv{L_r(\Gamma)}(t), \quad t>0.$$
From this, \eqref{eq:lower} and Example~\ref{env-Lpq} it follows that
$$\sigma_j^{-1}h_j^{-\frac{1}{p}} \lesssim h_j^{-\frac{1}{r}} \quad \mbox{for large enough } \, j \in \nat,$$
that is, $\bm{\sigma}^{-1}\bm{h}^{\frac{1}{r}-\frac{1}{p}} \in \ell_\infty$.\\

{\em Step 3}.\quad It remains to deal with the necessity of the
condition in (ii). 
Let $L \in \nat$  be such that $\, L > -1 + n(\frac{1}{p}-1)_+ - \sul(\bm{\sigma} \bm{h}^{1/p} \paren{n}^{1/p})$ and consider a $C^\infty$-function $\phi$ on $\rn$ such that 
\begin{itemize}
	\item there exist $C_1, C_2, C_3 > 0$ with $C_1<C_3<2C_1$, $\phi(x) \geq C_2$ for $| x |_\infty \leq C_1$ and $\phi(x) = 0$ for $| x |_\infty \geq C_3$,
	\item $\int_{\rn} { x^\beta \phi(x) \dint x } = 0$ whenever $\beta = (\beta_j)_{j=1}^n \in \non$ with $|\beta|_1 \leq L$.
\end{itemize}

Such a function exists (cf. \cite[Lemma 4.6]{C-Fa}). Moreover,
$$\sigma_j^{-1} h_j^{-\frac{1}{p}} \phi(2^j(\cdot-\gamma_0)), \quad j \in \nat, \quad \gamma_0 \, \mbox{ fixed in } \, \Gamma,$$
are (up to multiplicative constants) $(\bm{\sigma}\bm{h}^{\frac{1}{p}}\paren{n}^{\frac{1}{p}},p)_{K,L}$-atoms located at $Q_{j,m(j)}$, where $K$ is a fixed natural number satisfying $K > \lowind{\bm{\sigma}\bm{h}^{1/p} \paren{n}^{1/p}}$ and $m(j) \in \non$ is one of the possible $n$-tuples satisfying $|\gamma_0-2^{-j}m(j)|_\infty \leq 2^{-j-1}$. 

Let ${\bm b} := \gseq{{b}}{j}{\nat} \in \ell_q$ be any sequence of non-negative numbers. For any given $T \in \nat$ consider the sequence $\bm{b^T}$ formed by the terms of $\bm{b}$ until the order $T$ and by null terms afterwards and also consider
\begin{equation}
\label{gbT}
g^{\bm b^T}(x) := \sum_{j=1}^{T} { {b}_j
\sigma_{j \iota_0}^{-1} h_{j \iota_0}^{-\frac{1}{p}}
\phi(2^{j\iota_0}(x-\gamma_0)) } , \quad x \in \rn , 
\end{equation}
for the $\gamma_0$ fixed in $\Gamma$ considered above and with $\iota_0 \in \nat$ at our disposal.

From the atomic representation Proposition~\ref{atomicdecomposition} applied to $\Bsih$ and the above considerations, there exists $c'>0$ independent of $\bm{b}$ and $T$ such that
\begin{equation}
\| g^{\bm b^T} | \Bsih \| \leq c'\, \| {\bm b^T} | \ell_q \| \leq c' \, \| \bm{b} | \ell_q \| < \infty.
\label{gbT'}
\end{equation}
From our existence hypothesis and the fact that $g^{\bm b^T} \in C_0^\infty(\rn) \subset \SRn$, there also exists $c''>0$ independent of ${\bm b^T}$ and $T$ such that
$$\|g^{\bm b^T}|_\Gamma | L_r(\Gamma) \| \leq c''\, \| g^{\bm b^T} | \Bsih \|$$
and therefore 
\begin{equation}
\|g^{\bm b^T}|_\Gamma | L_r(\Gamma) \| \leq c'' c' \, \| \bm{b} | \ell_q \| < \infty.
\label{eq:ineqqr}
\end{equation}

For each $m \in \nat$ define
$$P_m := \{ x \in \rn : C_3 2^{-(m+1)\iota_0} < |x - \gamma_0|_\infty \leq C_1 2^{-m\iota_0} \}.$$
Since $C_1 < C_3 < 2C_1 \leq 2^{\iota_0}C_1$, then $C_1 2^{-(m+1)\iota_0} < C_3 2^{-(m+1)\iota_0} < C_1 2^{-m\iota_0}$ and therefore $P_m \not= \emptyset$ and $P_{m+1} \cap P_m = \emptyset$.
Observe that if $x \in P_m$, for some given $m\in\nat$, then
$$\phi(2^j(x-\gamma_0)) \geq C_2 \quad \mbox{if } \, 1 \leq j \leq m\iota_0 \quad \mbox{and}$$
$$\phi(2^j(x-\gamma_0)) = 0 \quad \mbox{if } \, j \geq (m+1)\iota_0.$$

We can then write that
\begin{eqnarray}
\label{eq:below}
\|g^{\bm b^T}|_\Gamma | L_r(\Gamma) \| & \geq & \left( \sum_{m=1}^T \int_{\Gamma \cap P_m} \Big( \sum_{j=1}^m b_j \sigma_{j \iota_0}^{-1} h_{j \iota_0}^{-\frac{1}{p}} C_2 \Big)^r \dint\mu(\gamma) \right)^\frac{1}{r} \nonumber \\
 & \geq & C_2 \left( \sum_{m=1}^T b_m^r \sigma_{m \iota_0}^{-r} h_{m \iota_0}^{-\frac{r}{p}} \mu(\Gamma \cap P_m) \right)^\frac{1}{r}
\end{eqnarray}
where, using the notation $Q_t(\gamma) := \{ x \in \rn : |x-\gamma|_\infty \leq t \}$,
$$\mu(\Gamma \cap P_m) =  \mu(Q_{C_1 2^{-m\iota_0}}(\gamma_0)) - \mu(Q_{C_3 2^{-(m+1)\iota_0}}(\gamma_0)).$$

Now we recall that $\Gamma$ is an $h$-set and therefore, besides the use of balls in Definition~\ref{defi-hset}(ii), we can also say that, given any $t_0>0$, there exist constants $a_0,a_1>0$ such that, for any $\gamma \in \Gamma$ and $t \in (0,t_0]$, $a_0 h(t) \leq \mu(Q_t(\gamma)) \leq a_1 h(t)$. By using Remark~\ref{rem-doubling}, the fact that $h$ is non-decreasing and the hypothesis $\upind{\bm{h}}<0$, it is then possible to choose $\iota_0 \in \nat$ large enough in order that
\begin{eqnarray*}
\mu(\Gamma \cap P_m) & \geq & a_0 h(C_1 2^{-m\iota_0}) - a_1 h(C_3 2^{-(m+1)\iota_0}) \\
 & \geq & a_2(\iota_0)h_{m\iota_0},
\end{eqnarray*}
where $a_2(\iota_0)$ can be chosen positive and independent of $m$. Using this in \eqref{eq:below} and conjugating with \eqref{eq:ineqqr} we get that
$$ (b_m \sigma_{m \iota_0}^{-1} h_{m \iota_0}^{\frac{1}{r}-\frac{1}{p}})_{m\in \nat} \in \ell_r.$$

Let $\bm{\beta} = \bm{b}$, $\bm{\alpha}= (\sigma_{m\iota_0}^{-1}
h_{m\iota_0}^{\frac1r-\frac{1}{p}})_{m\in \nat}$. We apply
Lemma~\ref{lemma-Landau} with $q_1=q$ and $q_2=r$ and obtain that the
subsequence $(\sigma^{-1}_{m\iota_0}
h_{m\iota_0}^{\frac1r-\frac1p})_m\in\ell_{v_r}$. The extension to $
\bm{\sigma}^{-1} \bm{h}^{\frac1r-\frac{1}{p}}\in\ell_{v_r}$ is a
consequence of Remark~\ref{subseq-adm}.
\end{proof}

\begin{corollary}\label{C-Bgamma-Lloc}
Let $0<p,q<\infty$, $\bm{\sigma}$ be an admissible sequence and $\Gamma$ be an $h$-set satisfying the porosity condition and such that $\upind{\bm{h}}<0$. The following two assertions are equivalent:
\begin{description}
	\item[(i)] $\Bsg$ exists as the trace space of $\Bsih$ in $L_p(\Gamma)$ and 
	$$\Bsg \hookrightarrow L_{\max(p,1)}(\Gamma);$$
	\item[(ii)] $\bm{\sigma}^{-1} \bm{h}^{-\left( \frac{1}{p}-1 \right)_+} \in \ell_{q'}$.
\end{description}
\end{corollary}

\begin{proof}
(ii) $\Rightarrow$ (i): This was already proved in \cite[Prop.~3.9]{Cae-env-h}, even without special restrictions on the $h$-set.

(i) $\Rightarrow$ (ii): Consider first the case $1\leq p <
  \infty$. Since $\Gamma$ is compact and $p\geq 1$, the hypothesis implies that $\Bsg \hookrightarrow L_1(\Gamma)$ and the result follows by applying Proposition~\ref{ex-trace} with $r=1$.

Consider now the case $0<p<1$. Then the result follows from Proposition~\ref{coro-1} with $r=1$, using its part (i) in the subcase $0<q\leq 1$ and its part (ii) in the subcase $1<q<\infty$.
\end{proof}

\begin{remark}\label{BGamma-Lloc}{
If we compare the above corollary with its $\rn$-counterpart
Proposition~\ref{sharp-B-Lloc}, the condition in case of $0<p\leq 1$ is quite
similar, unlike in the case $1<p<\infty$; this might be surprising at
first glance, but the reason is presumably hidden in our assumption or
conclusion, respectively, that the
trace spaces should exist.}\end{remark}

\begin{remark}\label{R-B-C}{
Note that in the criterion \eqref{B-C-1} for $\Bsg \hookrightarrow C(\Gamma)$ only the
smoothness parameter $\bm{\sigma}$ and the underlying fractal geometry
$\bm{h}$ are involved (and not the underlying $\rn$-space) -- as it should be;
the same applies to Corollary~\ref{C-Bgamma-Lloc}. For similar discussions
concerning the independence of the dimension of the underlying space we refer to
\cite[Sects.~18.9, 20.3]{T-Frac}.
}\end{remark}

\begin{example}\label{R-B-Lloc}
If $\Gamma\ $ is a $d$-set, $h(r)=r^d$, $0<d< n$, and
$\bm{\sigma}= \paren{s}$ with $s>0$, then
Corollary~\ref{C-Bgamma-Lloc} 
reads as $ \mathbb{B}^s_{p,q}(\Gamma) \hookrightarrow
L_{\max(p,1)}(\Gamma)$ if, and only if, $ s>d (\frac1p-1)_+ $,
$0<q<\infty$, or $\ s = d (\frac1p-1)_+$ when $\ 0<q\leq 1$. This obviously
goes well with the classical situation in $\rn$,  cf. \cite[Thm.~3.3.2]{SiT} for a
complete characterisation, and \cite[Sect.~20.3]{T-Frac} for the case of $d$-sets.
\end{example}

\begin{corollary}
\label{coro-1a}
Let $0<p,q<\infty$, $\bm{\sigma}$ be an admissible sequence and $\Gamma$ be an $h$-set satisfying the porosity condition and such that $\upind{\bm{h}}<0$.  Then $\Bsg$ exists as the trace space of $\Bsih$ in $L_p(\Gamma)$ if, and only if,
\[
\bm{\sigma}^{-1} \in \begin{cases}
\ell_{q'}, &\text{if}\quad 1\leq p<\infty \quad\text{or}\quad 0<q\leq p<1,\\
\ell_{v_p}, &\text{if}\quad  0<p< 1 \quad\text{and}\quad  0< p < q<\infty,
\end{cases} 
\]
where $\frac{1}{v_p}:=\frac{1}{p}-\frac{1}{q}$.
\end{corollary}

\begin{proof}
The sufficiency comes from Definition~\ref{defi-Bsg} (which is backed up by the discussion in \cite[Sect.~3.1]{Cae-env-h}), even without special restrictions on the $h$-set.

As to the necessity, for the first line of the brace it comes from Proposition~\ref{ex-trace} with $r=p$ (and the assumption $\upind{\bm{h}}<0$ is not needed here), while for the second line it comes from part (ii) of Proposition~\ref{coro-1} with $r=p$ (and the porosity condition is not needed here).
\end{proof}

\begin{remark}\label{striso}{
Recall that $\upind{\bm{h}}<0$ implies that $\mu$ is strongly isotropic in the sense of Remark~\ref{strong-iso-index} and that
$\bm{h} \in \ell_1$. In particular, this implies
(cf. \cite[Sect.~2.1]{Cae-env-h}) that if $\bm{\sigma}^{-1} \notin
\ell_{v_p}$ (with $0<p< 1$ and $0< p < q<\infty$), then
$\bm{\sigma}^{-1} \bm{h}^{\frac{1}{r}- \frac{1}{p}} \notin \ell_{v_r}$
for all $r \in [p,\min(q,1)]$ --- which had to be, otherwise the above result would contradict Definition~\ref{defi-Bsg}(ii).
}\end{remark}

\begin{remark}\label{dichotomy}{
In \cite[Conj.~3.21]{CaH-3} we conjectured that, instead of
$\upind{\bm{h}}<0$, an extra assumption like $\lim_{j \to \infty} h_j
\sigma_j^{v_p}=0$ would give the equivalence above in the case $0<p<
1$ and $0< p < q<\infty$ together with a so-called dichotomy
result. We have nothing new to add here as far as the dichotomy is
concerned, but we would like to draw the attention of the reader to
\cite[Rem.~3.22]{CaH-3}, which shows that a result like the one above
for $0<p< 1$ and $0< p < q<\infty$ cannot hold without extra assumptions. In particular, even if not knowing whether $\upind{\bm{h}}<0$ can be improved or not as an extra requirement, at least we know that something extra must be required in the mentioned case of the result above.
}\end{remark}

\begin{remark}\label{emb-Bsg-Lr-comp}{
{Bricchi} obtained in  
\cite[Thm.~4.3.2]{bricchi-diss} that for $\lowind{\sigma}>0$, $1\leq r\leq 
p<\infty$, $0<q<\infty$, and with some additional assumptions on $\Gamma$
and $h$, the embedding $\ \Bsg \hookrightarrow L_r(\Gamma)$ is compact and he estimated the corresponding entropy numbers by
\begin{equation}
e_{\whole{h_j^{-1}}}\left(\Bsg \hookrightarrow L_r(\Gamma)\right) \sim \
\sigma_j^{-1}, \quad j\in\nat. 
\label{ek-Bgamma}
\end{equation}
As already pointed out in Remark~\ref{R-Boyd}, $\lowind{\sigma}>0$
implies $\bm{\sigma}^{-1} \in \ell_v$ for arbitrary $v$, but we
conjecture  
that the embedding $\ \Bsg \hookrightarrow L_r(\Gamma)$ remains compact for
$\lowind{\sigma}=0$ as long as $\bm{\sigma}^{-1} \in \ell_{q'}$ and $q>1$,
that is, in particular, with $\sigma_j^{-1} \to 0$ for $j\to\infty$, in good agreement
with \eqref{ek-Bgamma}. 
}\end{remark}

\begin{remark}\label{R-trace-ak}{
We return to Remark~\ref{R-trace-ex} and the special setting {Triebel} studied in \cite[Cor.~7.21]{T-F3}. For $1<p=q<\infty$, and $\bm{\sigma}\paren{n}^{\frac1p} \bm{h}^{\frac1p} =
\paren{s}$ with $ 0<s\leq\frac{n}{p}$, the trace operator 
\[
\tr{\Gamma}: B^s_{p,p}(\rn) \to L_p(\Gamma)
\]
is not only compact, but one can also determine the asymptotic behaviour of its approximation
numbers: Under the additional assumption that $ \mu$ is strongly isotropic in the sense of
Remark~\ref{strongly-isotropic} and satisfies
\begin{equation*}
\sum_{j\geq J} 2^{-j p'(s-\frac{n}{p})} h_j^{\frac{p'}{p}} \sim 2^{-J
  p'(s-\frac{n}{p})} h_J^{\frac{p'}{p}}, \quad J\in\nat,
\label{T-ak}
\end{equation*}
then, denoting the inverse function of $\bm{h}$ by $H$, {Triebel} proved in \cite[Thm.~7.22]{T-F3} that 
\[
a_k\left(\tr{\Gamma}: B^s_{p,p}(\rn) \to L_p(\Gamma)\right) \sim \
k^{-\frac1p} H(k^{-1})^{s-\frac{n}{p}}, \quad k\in\nat. 
\]
Following the proof in \cite[Thm.~7.22]{T-F3} and our different
notation this can be rewritten as 
\[
a_{\whole{h_j^{-1}}}\left(\tr{\Gamma}: B^{\bm{\sigma}
  \bm{h}^{1/p}\paren{n}^{1/p}}_{p,p}\!(\rn) \to L_p(\Gamma)\right)
\sim \ \sigma_j^{-1}, \quad j\in\nat,
\]
assuming that $1<p<\infty$, $\bm{\sigma}\paren{n}^{\frac1p} \bm{h}^{\frac1p} =
\paren{s}$ with $ 0<s\leq\frac{n}{p}$, and $\sum_{j\geq J} \sigma_j^{-p'} \sim \sigma_J^{-p'}$,
$J\in\nat$. Though in a slightly different setting, the similarity to
\eqref{ek-Bgamma} is obvious.
}\end{remark}

\begin{example}\label{emb-in-Lp}
In case of $0<p\leq r\leq 1$, $0< q< \infty$, and $\upind{\bm{h}}<0$, 
we obtain from Proposition~\ref{coro-1} that
\begin{equation*}
\Bsg \hookrightarrow L_r(\Gamma)
\label{s-15}
\end{equation*}
if, and only if,
\begin{equation*}
\bm{\sigma}^{-1} \bm{h}^{\frac1r-\frac1p}\in \ell_{v_r^+}\quad \mbox{with}\quad
\frac{1}{v_r^+} = \left(\frac1r-\frac1q\right)_+\ . 
\label{s-16}
\end{equation*}
In particular, when $\Gamma$ is a $d$-set with $0<d<n$, and
$\bm{\sigma}=\paren{s}$, then this result reads as 
$$ \mathbb{B}^s_{p,q}(\Gamma) \hookrightarrow L_r(\Gamma) $$
if, and only if, $s-\frac{d}{p}\geq -\frac{d}{r}$ if $q\leq r$,
or $s-\frac{d}{p}> -\frac{d}{r}$ if $q> r$. In other words, we
have for the limiting case $s-\frac{d}{p}= -\frac{d}{r}$ that 
$$ \mathbb{B}^{d(\frac1p-\frac1r)}_{p,q}(\Gamma) \hookrightarrow L_r(\Gamma)
\quad \mbox{if, and only if,}\quad q\leq r.$$
This behaviour is well-known from the classical $\rn$-situation, see
\cite[Rem.~3.3.5]{SiT} or \cite[Thm.~11.4]{T-func} for the case $r\geq
1$, and \cite[Thm.~1.15]{HaSch} otherwise. There it is proved by some
interpolation argument involving Lorentz spaces. Note that one has to be quite careful with the
definition of Besov spaces when $0<p<1$ and $0<s\leq n(\frac1p-1)$, that
is, whether the corresponding spaces are defined as spaces of tempered
distributions or as subspaces of $L_p$. 
\end{example}

We finally come to embeddings between two such Besov spaces on some
$h$-set $\Gamma$. We begin with sufficient conditions for such an
embedding. 

\begin{theorem}
Let $\Gamma$ be an $h$-set, 
$0<p_1, p_2< \infty$, $0<q_1, q_2< \infty$, $\bm{\sigma}$ and
$\bm{\tau}$ be admissible sequences. Let $q^\ast$ be given by
\eqref{q-ast}. 
If $\lowind{\bm{\tau}}>0$ and
\begin{equation}
\bm{\sigma}^{-1} \bm{\tau} \bm{h}^{-\left(\frac{1}{p_1}-\frac{1}{p_2}\right)_+} \in
\ell_{q^\ast}\ ,
\label{lim_emb-2a}
\end{equation}
then there exist the trace spaces $\Bsge$ and $\Btgz$ and
\begin{equation}
\Bsge \ \hookrightarrow\ \Btgz\ .
\label{lim_emb-1a}
\end{equation}
\label{emb-BGamma-suff}
\end{theorem}

\begin{proof}
{\em Step 1}.\quad We consider the case $p_1\leq p_2$. Since in that case
\eqref{lim_emb-2a} ensures that
\[
\left[\bm{\sigma}\bm{h}^{\frac{1}{p_1}}\paren{n}^{\frac{1}{p_1}}\right]^{-1} \left[\bm{\tau}\bm{h}^{\frac{1}{p_2}}\paren{n}^{\frac{1}{p_2}}\right]\paren{n}^{\frac{1}{p_1}-\frac{1}{p_2}}=
\bm{\sigma}^{-1} \bm{\tau} \bm{h}^{-\left(\frac{1}{p_1}-\frac{1}{p_2}\right)_+} \in
\ell_{q^\ast}\ ,
\]
then, by Theorem~\ref{prop-Bsi-emb}, the embedding
\begin{equation}
\id_1: \Bsihe \to \Btauhz
\label{id_1-B-B}
\end{equation}
exists and is bounded. \\
On the other hand, following our discussion in Section~\ref{sect-2-1},
the hypothesis $\lowind{\bm{\tau}}>0$ guarantees the existence of the
trace space $\Btgz$, that is, the existence and boundedness of the
trace operator 
\[
\tr{\Gamma} : \Btauhz \to L_{p_2}(\Gamma)
\]
such that $\tr{\Gamma}\varphi=
\varphi\raisebox{-0.4ex}[1.1ex][0.9ex]{$|_{\Gamma}$}$ whenever
$\varphi\in\SRn$. Since $p_1\leq p_2$ and $\Gamma$ is compact, we
also have that the embedding
\[
\id_2: L_{p_2}(\Gamma)\to L_{p_1}(\Gamma)
\]
exists and is bounded. Therefore
\[
\id_2\circ \tr{\Gamma}\circ \id_1 : \Bsihe \to L_{p_1}(\Gamma)
\]
is a bounded linear operator such that, given any $\varphi\in\SRn$,
\[
\left(\id_2\circ \tr{\Gamma}\circ \id_1 \right)\varphi =
\varphi\raisebox{-0.4ex}[1.1ex][0.9ex]{$|_{\Gamma}$}\ .
\]
By Remark~\ref{R-unique} it must be the trace operator of $\Bsihe$ in
$L_{p_1}(\Gamma)$, hence there exists also 
\[\Bsge=\left(\id_2\circ
\tr{\Gamma}\circ \id_1 \right)\Bsihe.
\] 
Moreover, given $f\in\Bsge$, there is some $g\in\Bsihe $ such that $f=\left(\id_2\circ
\tr{\Gamma}\circ \id_1 \right)g$. But
\eqref{id_1-B-B} implies also $g\in\Btauhz$, thus also $f\in\Btgz$, and
\begin{align*}
\left\| f| \Btgz\right\| & = \inf_{\tr{\Gamma} g=f} \left\| g|
\Btauhz\right\|\\
&\leq \ c  \ \inf_{(\id_2\circ \tr{\Gamma}\circ \id_1)g=f} \  \left\| g|\Bsihe\right\|\\
& = \ c \left\| f  | \Bsge\right\|.
\end{align*}

{\em Step 2}.\quad We deal with the case $p_1>p_2$. Applying the
preceding step to $p_1=p_2$, we have 
\[\Bsge \hookrightarrow \mathbb{B}^{\bm{\tau}}_{p_1, q_2}(\Gamma)
\]
since \eqref{lim_emb-2a} reads then as $\bm{\sigma}^{-1} \bm{\tau} \in
\ell_{q^\ast}$. This also covers the existence of the trace
spaces. The hypothesis  $\lowind{\bm{\tau}}>0$ guarantees, as before,
the existence of the trace space $\Btgz$, and from the assumption
$p_1>p_2$ it follows that 
\[
\mathbb{B}^{\bm{\tau}}_{p_1, q_2}(\Gamma) \hookrightarrow \Btgz.
\]
This can be proved by using atomic representations for the
corresponding spaces on $\rn$, thus extending a related result in
\cite[Step~2 of the proof of Thm.~4.3.2]{bricchi-diss}. Putting the two previous
embeddings together we arrive at $\Bsge \hookrightarrow \Btgz$, as desired.
\end{proof}

\begin{remark}\label{R-assumpt-suff}{
From the above proof we see that the hypothesis $\lowind{\bm{\tau}}>0$
was only used to guarantee the existence of some trace spaces. We can
dispense with it if we, instead, assume that $\Btgz$ exists, and, in
case $p_1>p_2$, that also ${\mathbb{B}}^{\bm{\tau}}_{p_1,q_2}(\Gamma)$ exists.
}\end{remark}

\begin{remark}\label{suff-bricchi}{
Our result above extends a corresponding one by {Bricchi} in 
\cite[Thm.~4.3.1]{bricchi-diss} as far as the existence of trace
spaces is concerned, whereas {Bricchi}
obtained also compactness assertions there.
}\end{remark}

\begin{example}\label{emb-suff-d}
Let $0<d<n$, $\Gamma$ be a $d$-set, 
$0<p_1, p_2< \infty$, $0<q_1, q_2< \infty$, $s,t>0$. Then
\[
\mathbb{B}^s_{p_1,q_1}(\Gamma) \hookrightarrow
\mathbb{B}^t_{p_2,q_2}(\Gamma)\quad \text{if}\quad \begin{cases}
s-t\geq d\left(\frac{1}{p_1}-\frac{1}{p_2}\right)_+ &
\text{if}\ q_1\leq q_2, \\[1ex] s-t> d\left(\frac{1}{p_1}-\frac{1}{p_2}\right)_+ &
\text{if}\ q_1> q_2\ .
\end{cases}
\] 
This is an immediate consequence of
Theorem~\ref{emb-BGamma-suff}. We thus recover (part of) the
result of {Triebel} in \cite[Thm.~20.6]{T-Frac} where he also
dealt with the asymptotic behaviour of the entropy numbers of the
compact embedding $\mathbb{B}^s_{p_1,q_1}(\Gamma) \hookrightarrow
\mathbb{B}^t_{p_2,q_2}(\Gamma)$ when $s-t>
d\left(\frac{1}{p_1}-\frac{1}{p_2}\right)_+$ (including further
limiting cases of $t=0$, $p_i, q_i=\infty$). 
\end{example}

\begin{example}\label{emb-suff-d-Psi}
Similarly, if we consider $(d,\Psi)$-sets
  $\Gamma$, $0<d<n$, $\Psi$ an admissible function, and corresponding trace spaces
  $\mathbb{B}^{(s,\Psi)}_{p,q}(\Gamma)$, recall
  Remark~\ref{BDef-d-sets}, then we can compare our result with a
  corresponding one of {Moura} in \cite[Thm.~3.3.2]{moura-diss} where
  the focus was again on compactness assertions rather than on
  (limiting) continuous 
  embeddings. Explicating Theorem~\ref{emb-BGamma-suff} in that
  setting we obtain for $0<p_1, p_2< \infty$, $0<q_1, q_2< \infty$,
  $s,t>0$, that
\[
\mathbb{B}^{(s,\Psi)}_{p_1,q_1}(\Gamma) \hookrightarrow
\mathbb{B}^{(t,\Psi)}_{p_2,q_2}(\Gamma)\quad \text{if}\quad 
\left(2^{-j(s-t-d(\frac{1}{p_1}-\frac{1}{p_2})_+)}
\Psi(2^{-j})^{-(\frac{1}{p_1}-\frac{1}{p_2})_+} \right)_j \in \ell_{q^\ast}.
\] 
\end{example}

We study the necessity of the condition \eqref{lim_emb-2a}  for the embedding \eqref{lim_emb-1a} now.

\begin{theorem}
Let $\Gamma$ be an $h$-set satisfying the porosity condition with \eqref{T-1}, 
$0<p_1, p_2< \infty$, $0<q_1, q_2< \infty$, $\bm{\sigma}$ and
$\bm{\tau}$ be admissible sequences with 
\begin{equation}
\upind{\tau \bm{h}^{1/p_2}}<0
\label{E-3a}
\end{equation}
and
\begin{equation}
\lowind{\tau} > -\lowind{h}\left(\frac{1}{p_2}-1\right)_+ \ .
\label{E-3b}
\end{equation}
Let $q^\ast$ be given by \eqref{q-ast}. If the trace spaces $\Bsge$ and $\Btgz$ exist and 
\begin{equation}
\Bsge \ \hookrightarrow\ \Btgz,
\label{B3}
\end{equation}
then 
\begin{equation}
\bm{\sigma}^{-1} \bm{\tau} \bm{h}^{-\left(\frac{1}{p_1}-\frac{1}{p_2}\right)} \in
\ell_{q^\ast}\ .
\label{B4}
\end{equation}
\label{emb-BGamma-nec}
\end{theorem}

\begin{proof}
{\em Step 1}.\quad We would like to start by remarking that the assumption that ${\mathbb{B}}^{\bm{\tau}}_{p_2,q_2}(\Gamma)$ exists is not really necessary, as it comes as a consequence of \eqref{T-1} and \eqref{E-3b}, due to Corollary~\ref{coro-1a} and (the end of) Remark~\ref{R-Boyd}. The latter also guarantees, together with \eqref{E-3a}, that $\bm{\tau}^{-1}\bm{h}^{-1/p_2}\not\in\ell_\infty$, therefore
\begin{equation}
\bm{\tau}^{-1}\bm{h}^{-1/p_2}\not\in\ell_{q_2'} .
\label{B5}
\end{equation}
Hence, by Theorem~\ref{below 2},  $\Btgz$ possesses a non-trivial growth envelope. We return to the example $g^{{\bm b^T}}$ given by   \eqref{gbT}, 
\begin{equation*}
g^{\bm b^T}(x) = \sum_{r=1}^{T} { {b}_r
\sigma_{r \iota_0}^{-1} h_{r \iota_0}^{-\frac{1}{p_1}}
\phi(2^{r\iota_0}(x-\gamma_0)) } , \quad x \in \rn , 
\label{B6}
\end{equation*}
where now the sequence $\bm{b}=(b_r)_{r=1}^\infty$ of non-negative numbers belongs to $\ell_{q_1}$. According to \eqref{qnormBG} and \eqref{gbT'} we have
\begin{equation*}
\Vert g^{\bm{b}^T}\mid \Bsge \Vert \leq c_1^{-1}\, \Vert
{\bm{b}}\mid \ell_{q_1}\Vert
\label{s-11}
\end{equation*}
(recall that we assume that $\Bsge$ exists, see also the discussion in Section~\ref{sect-2-1}). On the other hand, similarly as in Step~3 of the proof of Proposition~\ref{coro-1} (or directly from \cite[Lemma~3.3]{Cae-env-h}), if $\iota_0\in\nat$ is chosen large enough we have
\begin{equation}
(g^{\bm b^T})^{*,\mu}(c_2 h_{k \iota_0})
 \geq c' \sum_{r=1}^{k}
 { {b}_r \sigma_{r \iota_0}^{-1} h_{r \iota_0}^{-\frac{1}{p_1}} } , \quad k
 \in \nat,
\label{s-12}
\end{equation}
with $c_2\in (0,1]$. We apply Theorem~\ref{below 2} to $\Btgz$ and can
thus estimate, also due to our hypothesis \eqref{B3},
\begin{align}
\left\| \bm{b} | \ell_{q_1}\right\| &\geq c_1 \left\| g^{\bm{b}^T} |\Bsge \right\| \nonumber\\
&\geq c_3 \left\| g^{\bm{b}^T} | \Btgz \right\| \nonumber\\
&\geq c_4 \left(\int_0^\varepsilon
\left[\frac{(g^{\bm b^T})^{*,\mu}(t)}{\Psi^{\bm{\tau}}_{p_2, q_2'}(t)}\right]^{q_2} 
\mu_{\Psi^{\bm{\tau}}} (\dint t)
\right)^{1/q_2}    
\label{B9}
\end{align}
where $\Psi^{\bm{\tau}}_{p_2, q_2'}$ is given by \eqref{Psi-ru} or \eqref{Psi-ru'} with $r=p_2$, $u=q_2'$ and $\bm{\sigma}$ replaced by $\bm{\tau}$, and $\mu_{\Psi^{\bm{\tau}}}$ stands for the corresponding Borel measure $ \mu_{\mathsf G} $ as explained before Definition~\ref{defi-envg}.\\

{\em Step 2}.\quad We first deal with the case $q_2>1$. We proceed from \eqref{B9}, where we can write, equivalently, the last expression as
\begin{equation}
\label{B10}
\left(\int_0^\varepsilon
\left[\frac{(g^{\bm b^T})^{*,\mu}(t)}{\Psi^{\bm{\tau}}_{p_2, q_2'}(t)^{q_2'}}\right]^{q_2} 
t^{-\frac{q_2'}{p_2}-1} \Lambda_{\bm{\tau}}(t^{-\frac1n})^{-q_2'} \dint t
\right)^{1/q_2},    
\end{equation}
where $\Lambda_{\bm{\tau}}$ is the same as in \eqref{Lambda} with $\bm{N}={\bm h}^{-1}$ (see also \cite[Rem.~4.15]{C-Fa}). Using \eqref{s-12}, the admissibility of the sequences $\bm{\tau}$ and ${\bm h}$, and discretising \eqref{Psi-ru} (see also \cite[Lemma~2.6, Prop.~2.7]{C-Fa}), 
\eqref{B10} is bounded from below by
\begin{equation}
\Big(\sum_{k=k_0}^\infty \Big(\sum_{r=1}^k b_r \sigma_{r \iota_0}^{-1} h_{r\iota_0}^{-\frac{1}{p_1}}\Big)^{q_2} \Big(\sum_{r=1}^{k\iota_0} h_r^{-\frac{q_2'}{p_2}} \tau_r^{-q_2'}\Big)^{-q_2} h_{k\iota_0}^{-\frac{q_2'}{p_2}} \tau_{k\iota_0 }^{-q_2'} \Big)^{\frac{1}{q_2}}
\label{B11}
\end{equation}
up to a constant factor, where $k_0\in\nat$ is chosen sufficiently large (depending on $\bm{h}$, $\iota_0$ and $\varepsilon$). Observe now that \eqref{E-3a} implies that there exists $k_1>1$ and $j_0\in\nat$ such that
\begin{equation}
\tau_{j+1}^{-1} h_{j+1}^{-\frac{1}{p_2}} \geq k_1\ \tau_{j}^{-1} h_{j}^{-\frac{1}{p_2}}, \quad j\geq j_0,
\label{B12}
\end{equation}
from which follows, using also the admissibility of $\bm{\tau}$ and $\bm{h}$, that
\begin{equation}
\Big(\sum_{r=1}^{k \iota_0 } h_r^{-\frac{q_2'}{p_2}} \tau_r^{-q_2'}\Big)^{-1}  \sim 
h_{k\iota_0 }^{\frac{q_2'}{p_2}} \tau_{k\iota_0}^{q_2'}
\label{B13}
\end{equation}
(cf. also Remark~\ref{strong-iso-index}). Therefore, putting together \eqref{B9}-\eqref{B11}, we can write
\begin{align*}
\left\| \bm{b} | \ell_{q_1}\right\| &\geq \ c_5\ 
\Big(\sum_{k=k_0}^\infty b_k^{q_2} \sigma_{k\iota_0 }^{-q_2} h_{k\iota_0}^{-\frac{q_2}{p_1} + \frac{q_2'}{p_2}(q_2-1)} \tau_{k\iota_0}^{q_2'(q_2-1)}\Big)^{\frac{1}{q_2}}\\
&= \ c_5\ \Big(\sum_{k=k_0}^\infty b_k^{q_2} \sigma_{k\iota_0 }^{-q_2} h_{k\iota_0}^{q_2(\frac{1}{p_2}-\frac{1}{p_1})}
 \tau_{k\iota_0}^{q_2}\Big)^{\frac{1}{q_2}}.
\end{align*} 
Let $\bm{\beta}=\bm{b}$, $\bm{\alpha}=\left(\sigma_{k \iota_0}^{-1} h_{k\iota_0}^{\frac{1}{p_2}-\frac{1}{p_1}}\tau_{k\iota_0}\right)_k $. Then Lemma~\ref{lemma-Landau} yields
\[
\left(\sigma_{k \iota_0}^{-1} h_{k\iota_0}^{\frac{1}{p_2}-\frac{1}{p_1}}\tau_{k\iota_0}\right)_k \in \ell_{q^\ast}
\]
and the result \eqref{B4} follows by Remark~\ref{subseq-adm}.\\

{\em Step 3}.\quad We turn to the case $0<q_2\leq 1$. Then $\Psi^{\bm{\tau}}_{p_2, q_2'} = \Psi^{\bm{\tau}}_{p_2, \infty}$ and a discretisation argument (see e.g. \cite[Lemma~2.6, Prop.~2.7, Rem.~2.8]{C-Fa}) gives us, with a small enough positive constant $C\leq c_2$,
\[
\Psi^{\bm{\tau}}_{p_2, \infty}(C h_k) \sim \sup_{1\leq i\leq k} \tau_i^{-1} h_i^{-\frac{1}{p_2}}, \quad k\in\nat.
\] 
On the other hand, similarly to \eqref{B12} and \eqref{B13}, assumption \eqref{E-3a} and the admissibility of $\bm{\tau}$ and $\bm{h}$ implies that
\[
\sup_{1\leq i\leq k} \tau_i^{-1} h_i^{-\frac{1}{p_2}} \sim \tau_k^{-1} h_k^{-\frac{1}{p_2}} , \quad k\in\nat.
\]
Denote by $C_1$ and $C_2$ positive constants such that 
\begin{equation}
C_1 \tau_k^{-1} h_k^{-\frac{1}{p_2}} \leq \Psi^{\bm{\tau}}_{p_2, \infty}(C h_k) \leq C_2 \tau_k^{-1} h_k^{-\frac{1}{p_2}}, \quad k\in\nat.
\label{B14}
\end{equation}
Now we should be careful that $\iota_0$ has been chosen large enough in such a way that also
\begin{equation}
\frac{\tau_{(k+1)\iota_0}^{-1} h_{(k+1)\iota_0}^{-\frac{1}{p_2}}}{\tau_{k\iota_0}^{-1} h_{k\iota_0}^{-\frac{1}{p_2}}} \geq \ 2\frac{C_2}{C_1}, \quad k\in\nat.
\label{B15}
\end{equation}
This is clearly possible, due again to \eqref{E-3a}. We now discretise the right-hand side of \eqref{B9} so that, up to a constant factor, it is bounded from below by
\begin{equation}
\Big(\sum_{k=k_0}^\infty \Big(\sum_{r=1}^k b_r \sigma_{r \iota_0}^{-1} h_{r\iota_0}^{-\frac{1}{p_1}}\Big)^{q_2} \tau_{k\iota_0}^{q_2} h_{k\iota_0}^{\frac{q_2}{p_2}} \mu_{\Psi^{\bm\tau}}\left(\left[C h_{(k+1)\iota_0}, Ch_{k\iota_0}\right]\right)\Big)^{\frac{1}{q_2}}.
\label{B16}
\end{equation}
Observing that
\begin{align*}
\mu_{\Psi^{\bm\tau}}\left(\left[C h_{(k+1)\iota_0}, Ch_{k\iota_0}\right]\right) & = -\log \Psi^{\bm{\tau}}_{p_2,\infty}(C h_{k\iota_0}) + \log \Psi^{\bm{\tau}}_{p_2,\infty}(C h_{(k+1)\iota_0})\\
& = \log \frac{\Psi^{\bm{\tau}}_{p_2,\infty}(C h_{(k+1)\iota_0})}{\Psi^{\bm{\tau}}_{p_2,\infty}(C h_{k\iota_0})} \\
&\geq \log \frac{C_1 \tau_{(k+1)\iota_0}^{-1} h_{(k+1)\iota_0}^{-\frac{1}{p_2}}}{C_2 \tau_{k\iota_0}^{-1} h_{k\iota_0}^{-\frac{1}{p_2}}} \ \geq \ 1,
\end{align*}
where we used first \eqref{B14} and afterwards \eqref{B15}, we can further estimate \eqref{B16} from below by
\[
\Big(\sum_{k=k_0}^\infty b_k^{q_2} \sigma_{k\iota_0}^{-q_2} h_{k\iota_0}^{-\frac{q_2}{p_1}} \tau_{k\iota_0}^{q_2}h_{k\iota_0}^{\frac{q_2}{p_2}} \Big)^{\frac{1}{q_2}},
\]
so that, returning to \eqref{B9} we get
\[
\left\| \bm{b} | \ell_{q_1}\right\| \geq c_5\ \Big(\sum_{k=k_0}^\infty b_k^{q_2} \sigma_{k\iota_0}^{-q_2} h_{k\iota_0}^{q_2\left(\frac{1}{p_2}-\frac{1}{p_1}\right)} \tau_{k\iota_0}^{q_2} \Big)^{\frac{1}{q_2}}.
\]
We are now in the same situation as at the end of Step~2, so the last part of the proof follows as there.
\end{proof}

\begin{remark}\label{Thms_nec_suff}{
Analysing the proof we see that in the case $q_2>1$ we could have managed with weaker hypotheses, namely replacing $\upind{\tau \bm{h}^{1/p_2}}<0$, cf. \eqref{E-3a}, by $\bm{\tau}^{-1} \bm{h}^{-1/p_2} \not\in\ell_{q_2'}$, cf. \eqref{B5}, and 
\[\sum_{r=1}^{k \iota_0 } h_r^{-\frac{q_2'}{p_2}} \tau_r^{-q_2'}  \sim 
h_{k\iota_0 }^{-\frac{q_2'}{p_2}} \tau_{k\iota_0}^{-q_2'}, \quad k\in\no,
\] 
cf. \eqref{B13}. In the argument presented above these two assertions are consequences of our assumption \eqref{E-3a}, but what is really needed in Step~2 are \eqref{B5} and \eqref{B13} only. Note that \eqref{B13} is connected with requirements of strong isotropicity for the measure $\mu$ corresponding to $h$ and $\Gamma$, as already pointed out in Remark~\ref{strong-iso-index}.
}\end{remark}

In view of the fact that \eqref{E-3b} and \eqref{T-1}
immediately imply $\lowind{\tau}>0$, and \eqref{lim_emb-2a} and \eqref{B4} coincide for $p_1\leq p_2$, 
we can combine
Theorems~\ref{emb-BGamma-suff} and \ref{emb-BGamma-nec} in that case.

\begin{corollary}
Let $\Gamma$ be an $h$-set satisfying the porosity condition with \eqref{T-1}, 
$0<p_1\leq p_2< \infty$, $0<q_1, q_2< \infty$, and $q^\ast$ be given
by \eqref{q-ast}.  Let $\bm{\sigma}$ and
$\bm{\tau}$ be admissible sequences with \eqref{E-3a}, \eqref{E-3b}. Then the trace spaces $\Bsge$ and $\Btgz$ exist and
\begin{equation*}
\Bsge \ \hookrightarrow\ \Btgz,
\end{equation*}
if, and only if,
\begin{equation}
\bm{\sigma}^{-1} \bm{\tau} \bm{h}^{-\left(\frac{1}{p_1}-\frac{1}{p_2}\right)} \in
\ell_{q^\ast}\ .
\label{no+}
\end{equation}
\label{emb-BGamma-iff}
\end{corollary}

\begin{remark}\label{gap_p_1_p_2}{Unlike in case of $p_1\leq p_2$, there is an obvious gap for $p_1>p_2$ between the sufficient condition \eqref{lim_emb-2a} in Theorem \ref{emb-BGamma-suff} and the necessary one \eqref{B4} in Theorem \ref{emb-BGamma-nec}  in view of the embedding \eqref{lim_emb-1a}. We are not yet able to close it. But in view of similar observations in the classical case (embeddings of Besov spaces on a bounded domain $\Omega$), or special cases studied before, we claim that \eqref{lim_emb-2a} is the right one, not \eqref{B4}. In other words, we conjecture that Corollary \ref{emb-BGamma-iff} remains true irrespective of the relation between $p_1$ and $p_2$ if we replace \eqref{no+} by \eqref{lim_emb-2a}.
}
\end{remark}

We end this paper with explicating Corollary~\ref{emb-BGamma-iff} in
case of $d$- and $(d,\Psi)$-sets, recall also Remarks~\ref{example-ball} and \ref{BDef-d-sets}.

\begin{example}\label{emb-d-sets-iff}
Let $0<d<n$, $\Gamma$ be a $d$-set, 
$0<p_1\leq p_2< \infty$, $0<q_1, q_2< \infty$, $s>0$,
$d\left(\frac{1}{p_2}-1\right)_+<t<\frac{d}{p_2}$. 
Then
\[
\mathbb{B}^s_{p_1,q_1}(\Gamma) \hookrightarrow
\mathbb{B}^t_{p_2,q_2}(\Gamma)\quad \text{if, and only if,}\quad \begin{cases}
s-t\geq d\left(\frac{1}{p_1}-\frac{1}{p_2}\right) &
\text{if}\ q_1\leq q_2, \\[1ex]  s-t> d\left(\frac{1}{p_1}-\frac{1}{p_2}\right) &
\text{if}\ q_1> q_2\ .
\end{cases}
\] 
Similarly, if we consider $(d,\Psi)$-sets $\Gamma$, $0<d<n$, $\Psi$ an admissible function, and corresponding trace spaces
  $\mathbb{B}^{(s,\Psi)}_{p,q}(\Gamma)$, recall
  Remark~\ref{BDef-d-sets}, we get for $0<p_1\leq p_2< \infty$, $0<q_1, q_2< \infty$,
  $s>0$, $d\left(\frac{1}{p_2}-1\right)_+<t<\frac{d}{p_2}$, that
\[
\mathbb{B}^{(s,\Psi)}_{p_1,q_1}(\Gamma) \hookrightarrow
\mathbb{B}^{(t,\Psi)}_{p_2,q_2}(\Gamma)
\]
\text{if, and only if,}
\[ 
\left(2^{-j(s-t-d(\frac{1}{p_1}-\frac{1}{p_2}))}
\Psi(2^{-j})^{-(\frac{1}{p_1}-\frac{1}{p_2})} \right)_j \in \ell_{q^\ast}.
\] 
Here we have also used \eqref{index-ex} again.
\end{example}

   
{\bf Acknowledgement.} 
A.M. Caetano was partially supported by Portuguese funds through the CIDMA - Center for Research and Development in Mathematics and Applications, and the Portuguese Foundation for Science and Technology (``FCT - Funda\c c\~ao para a Ci\^encia e a Tecnologia''), within project \linebreak UID/MAT/04106/2013. Dorothee D. Haroske was partially supported by the DFG Heisenberg fellowship HA 2794/1-2. \\
It is our pleasure to thank our colleagues in the Department of
Mathematics at the University of Coimbra for their kind hospitality during our stays there.


\providecommand{\bysame}{\leavevmode\hbox to3em{\hrulefill}\thinspace}
\providecommand{\MR}{\relax\ifhmode\unskip\space\fi MR }
\providecommand{\MRhref}[2]{%
  \href{http://www.ams.org/mathscinet-getitem?mr=#1}{#2}
}
\providecommand{\href}[2]{#2}

\end{document}